\documentclass[11pt, a4paper]{article}
\usepackage{graphicx} 
\usepackage{geometry}

\usepackage[pdfencoding=unicode,
            psdextra,
            breaklinks=true,
            colorlinks,
            citecolor=blue,
            urlcolor=blue,
            linkcolor=blue]{hyperref} 

\usepackage[backend=biber, maxalphanames=6, style=alphabetic, maxbibnames=10]{biblatex}
\AtEveryBibitem{%
	\clearfield{isbn}
	\clearfield{issn}
	\clearfield{number}
	\clearfield{volume}
	\clearfield{series}
	\clearlist{language}
	\clearfield{url}
	\clearfield{pages}
	\clearfield{note}
}

\usepackage[T1]{fontenc}
\usepackage[english]{babel}
\usepackage{csquotes}

\usepackage{enumitem}

\usepackage{mathtools}
\DeclarePairedDelimiterX{\norm}[1]{\lVert}{\rVert}{#1}
\DeclarePairedDelimiterX{\inner}[2]{\langle}{\rangle}{\,#1,\,#2}
\DeclarePairedDelimiterX{\abs}[1]{\lvert}{\rvert}{#1}
\DeclarePairedDelimiterX{\seminorm}[1]{[}{]}{#1}

\usepackage{amsmath}
\usepackage{amsthm} 
\newtheorem{theorem}{Theorem}[section]
\newtheorem{lemma}[theorem]{Lemma}

\newtheorem{proposition}[theorem]{Proposition}
\newtheorem{corollary}[theorem]{Corollary}

\newtheorem{remark}[theorem]{Remark}

\usepackage{amssymb}

\renewbibmacro{in:}{}

\DeclareFieldFormat*{title}{\mkbibemph{#1}}
\DeclareFieldFormat*{journaltitle}{#1}
\DeclareFieldFormat*{booktitle}{#1}
\renewbibmacro{in:}{%
  \ifentrytype{article}
    {}
    {\printtext{\bibstring{in}\intitlepunct}}}
\DeclareFieldFormat[article,periodical]{volume}{\mkbibbold{#1}}
\DeclareFieldFormat[article,periodical]{number}{\bibstring{number}~#1}
\renewbibmacro*{journal+issuetitle}{%
  \usebibmacro{journal}%
  \setunit*{\addspace}%
  \iffieldundef{series}
    {}
    {\newunit
     \printfield{series}%
     \setunit{\addspace}}%
  \printfield{volume}%
  \setunit{\addspace}%
  \usebibmacro{issue+date}%
  \setunit{\addcolon\space}%
  \usebibmacro{issue}%
  \setunit{\addcomma\space}%
  \printfield{number}%
  \setunit{\addcomma\space}%
  \printfield{eid}%
  \newunit}
\DeclareFieldFormat[article,periodical]{pages}{#1}
\DeclareFieldFormat{mrnumber}{%
  MR\addcolon\space
  \ifhyperref
    {\href{http://www.ams.org/mathscinet-getitem?mr=#1}{\nolinkurl{#1}}}
    {\nolinkurl{#1}}}
\renewbibmacro*{doi+eprint+url}{%
  \iftoggle{bbx:doi}
    {\printfield{doi}}
    {}%
  \newunit\newblock
  \printfield{mrnumber}%
  \newunit\newblock
  \iftoggle{bbx:eprint}
    {\usebibmacro{eprint}}
    {}%
  \newunit\newblock
  \iftoggle{bbx:url}
    {\usebibmacro{url+urldate}}
    {}}
\newcommand{\bel}{\begin{equation} \label}
\newcommand{\ee}{\end{equation}}

\def\epsilon {\varepsilon}
\newcommand{\R}{\mathbb{R}}
\renewcommand{\leq}{\leqslant}
\renewcommand{\geq}{\geqslant}

\title{Recovery of nonlinear material parameters in a quasilinear Lam\'e system}

\begin{document}

\author{David Johansson$^1$ \and Yavar Kian$^2$}
\date{}

\maketitle

\begin{abstract} We investigate the inverse problem of determining nonlinear elastic material parameters from boundary stress measurements corresponding to prescribed boundary displacements. The material law is described by a nonlinear, space-independent elastic tensor depending on both the displacement and the strain, and gives rise to a general class of quasilinear Lam\'e systems. We prove the unique and stable recovery of a wide class of space-independent nonlinear elastic tensors, including the identification of two nonlinear isotropic Lam\'e moduli as well as certain anisotropic tensors. The boundary measurements are assumed to be available at a finite number of boundary points and, in the isotropic case, at a single point. Moreover, the measurements are generated by boundary displacements belonging to an explicit class of affine functions. The analysis is based on structural properties of nonlinear Lam\'e systems, including asymptotic expansions of the boundary stress and tensorial calculus.

{

}
\end{abstract}

\renewcommand{\thefootnote}{\fnsymbol{footnote}}
\footnotetext{\hspace*{-5mm} 
\begin{tabular}{@{}r@{}p{16cm}@{}} 
& Manuscript last updated: January 22, 2026.\\
$~^1$ 
& Aarhus University, Denmark. E-mail: \href{mailto:johansson@math.au.dk}{\texttt{johansson@math.au.dk}}.\\
$~^2$
& Univ Rouen Normandie, CNRS, Laboratoire de Math\'{e}matiques Rapha\"{e}l Salem, UMR 6085, F-76000 Rouen, France.  E-mail: \href{mailto:yavar.kian@univ-rouen.fr}{\texttt{yavar.kian@univ-rouen.fr}}\end{tabular}}

\section{Introduction}

In linear elasticity, the relation between the stress tensor and the strain tensor, also called the constitutive equation, is given by Hooke's law, which describes the deformation of a solid object under the action of applied forces. This leads to a linear Lam\'e system whose solutions correspond to the displacement of a point in the solid. This model is no longer valid for many real-world materials, where the relationship between stress and strain is not proportional and the stiffness of the material changes with deformation (see \cite{Cia,LaLif} for more details). Such properties are described by different nonlinear equations with many applications, including soft tissue modeling \cite{bio1,bio2}, cell mechanics \cite{bio3}, the aerospace industry \cite{Ae}, and polymer industries \cite{pol}.

In all of the above-mentioned applications, the nonlinear relation between the stress and the strain tensor is crucial to understanding properties of the corresponding materials. This relation can be represented by complex functions encoded into an elastic or stiffness tensor that  depends both on the 
displacement and the strain tensor, leading to a quasilinear Lam\'e system. In the present article, we study the determination of such a class of nonlinear elastic tensors from measurements of the stress generated by boundary displacements. 

Let \( \Omega\subset\mathbb{R}^{n} \), \( n\geq 2 \) be an open domain. We consider on the domain \(\Omega\) the stress tensor $\sigma$ given by 
\bel{law}\sigma(u)=C(u,\epsilon(u)):\epsilon(u),\ee
where \( u\colon \Omega\to\mathbb{R}^{n} \) denotes the displacement, \(\epsilon(u)=\frac{1}{2}(\nabla u + \nabla u^{T}) \) the strain tensor,
\(S_n(\mathbb R)\) the set of symmetric matrix defined by $S_n(\mathbb R):=\{A\in\mathbb R^{n\times n}:\ A^T=A\},$  and \( C\colon \mathbb{R}^{n}\times S_n(\mathbb R)\to \mathbb{R}^{n\times n\times n\times n} \) the elastic or stiffness tensor.   The equilibrium conditions without any body forces leads to the following  quasilinear Lam\'e system
\begin{equation}\label{eq-intro-quasilinear}
\begin{cases}
	\operatorname{div}(C(u,\epsilon(u)):\epsilon(u)) = 0&\text{in }\Omega, \\
	u = g&\text{on }\partial\Omega,
\end{cases}
\end{equation}
where   \( g\colon\partial\Omega\mapsto\mathbb{R}^{n} \) denotes the displacement imposed at the boundary. Observe that, the material parameters in \eqref{law} is given by the elastic tensor $C(u,\epsilon(u))$ which depends both on the displacement $u$ and the strain tensor $\epsilon(u)$. In the present article, we study the inverse problem of determining the nonlinear term $C(u,\epsilon(u))$ from boundary stress measurements generated by some affine class of boundary displacements $g$. We study both the uniqueness and the stability issue for this problem.

In the last couple of decades, significant attention has been devoted to identification of an elastic tensor from measurements of the  stresses at the boundary of the medium.
Most of the existing literature considered the determination of space-dependent Lam\'e parameters of a linear isotropic system from the knowledge of the corresponding displacement-to-traction map \cite{EsRa,GG,IY15,NU93,NU94,iuy12} and the identification of  inclusion or cracks for more general class of Lam\'e systems \cite{ebgh25,ABMa20,ABMa23,IkIt,IkIt1,NUW03,NUW05} (see also \cite{U99} for an overview).
Similar problems with more general elastic tensors were investigated in the time domain \cite{CNO,CDKU,DNZ,DUV,MaRa,Ra}.

To our best knowledge, most of the existing literature on identification of  nonlinear parameters in  Lam\'e systems are restricted to the time domain \cite{DUW,UZ,UZ24}.
For the recovery of Lam\'e parameters in a quasilinear elliptic system, we are only aware of \cite{NaSu} which establishes a local uniqueness result by reduction to the corresponding local uniqueness results for linear isotropic elasticity \cites{NU94}.

This contrasts with scalar elliptic equations, where the identification of nonlinear terms have been studied with a diversity of results on both semilinear terms \cite{FKO,FO,IN,IS,LLYS,jns25,Nu23,ST} and quasilinear terms \cite{Ca67,CFKKU,CLLO,Ki23,kian24,KKU,MU,Nu,SU97}.
The main goal of this work, is to bridge a gap in the theory of elasticity by proving that one can identify nonlinear material parameters, that depend both on the displacement and the strain tensor, for stationary Lam\'e systems.

The main results of this article are presented in Section \ref{sec2}. We separately address the identification of the dependence of the nonlinear parameter $C(u,\epsilon(u))$ on the displacement $u$ and on the strain tensor $\epsilon(u)$. To the best of our knowledge, these are the first results concerning the determination of general nonlinear material parameters in the Lam\'e system that depend simultaneously on the displacement and the strain tensor. These findings extend previous observations established for scalar elliptic equations \cite{CFKKU,Ki23,kian24,KKU,MU}. Moreover, we introduce additional restrictions on the data, refining those employed in earlier works.

Our analysis is based on precise asymptotic characterizations of boundary stress measurements combined with multilinear tonsorial calculus. These asymptotics are obtained via an explicit linearization procedure together with a detailed representation of the boundary stress. The methodology is informed by the linearization techniques developed in \cite{Is1,KLU,SU97}, and it extends earlier analysis for scalar elliptic equations \cite{Ca67,CFKKU,CLLO,Ki23,kian24,KKU,MU,Nu,SU97} to the Lam\'e system with general tensor-valued quasilinear terms. To the best of our knowledge, this work constitutes the first application of the infinite-order linearization method, introduced in \cite{KLU}, to the identification of quasilinear coefficients in the stationary Lam\'e system.

\section{Main results}\label{sec2}

Throughout the article, we impose the following assumptions:
\begin{enumerate}[label=A\arabic*)]
	\item\label{assumption-domain} \( \Omega\subseteq\mathbb{R}^{n} \), \( n\geq 2 \) is an open domain with \(  C^2 \) boundary,
	\item\label{assumption-nonlinearity} \( C\in C^{k+1}(\mathbb{R}^{n}\times S_n(\mathbb R), \mathbb{R}^{n\times n\times n\times n}) \) with $k\in\mathbb N$,
	\item\label{assumption-ellipticity} The tensor \( C(\lambda,0) \) is strongly elliptic for every \( \lambda\in\mathbb{R}^{n} \), i.e. for every \( \lambda\in \mathbb{R}^n \) there exists \( \kappa(\lambda)\in(0,+\infty) \) such that
\begin{equation*}
	\xi^{T}\eta:C(\lambda,0):\xi^{T}\eta=\sum_{i,j=1}^n\sum_{k,\ell=1}^nC_{ijk\ell}(\lambda,0)\xi_{\ell}\eta_{k}\xi_{j}\eta_{i}  \geq \kappa(\lambda)\norm{\xi}_{\mathbb{R}^{n}}^{2}\norm{\eta}_{\mathbb{R}^{n}}^{2},
\end{equation*}
for any \( \xi,\eta\in\mathbb{R}^{n} \).
	\item\label{assumption-symmetry} For every $(\lambda,\eta)\in \mathbb{R}^n\times S_n(\mathbb R)$ the tensor $C(\lambda,\eta)$ satisfies the symmetries
\begin{equation*}
	C_{ijk\ell}(\lambda,\eta) = C_{jik\ell}(\lambda,\eta) = C_{ij\ell k}(\lambda,\eta),\quad i,j,k,\ell=,\ldots,n,
\end{equation*}
and $C(\lambda,0)$ additionally satisfies the symmetry
 \begin{equation*}
 	C_{ijk\ell}(\lambda,0) = C_{k\ell ij}(\lambda,0),\quad i,j,k,\ell=,\ldots,n.
 \end{equation*}
\end{enumerate}
Under the above assumptions, we prove in Proposition \ref{lemma-solution-map} that,  for $p\in(n,+\infty)$ and  $g\in W^{2-1/p,p}(\partial\Omega, \mathbb{R}^{n})$ sufficiently close to a constant vector $\lambda\in\mathbb R^n$, problem  \eqref{eq-intro-quasilinear} admits a solution 
$u\in W^{2,p}(\Omega,\mathbb R^n)$.
Moreover, the solution \( u \) is unique within some possibly small neighborhood \( V_{\lambda}\subset W^{2,p}(\Omega,\mathbb{R}^{n}) \) of the constant \( \lambda \); see Section \ref{sec3} for more details.
In contrast to scalar elliptic equations, there is no comparison principle available for guaranteeing the global uniqueness of such solutions.
The local uniqueness in the neighborhood \( V_{\lambda} \) enables the construction of a local displacement-to-traction map, analogous to the Dirichlet-to-Neumann map in the classical Calderón problem.
Let $\lambda\in\mathbb R^n$ and, for \( \delta>0 \), denote by \( U_{\delta} \) the set defined by
\begin{equation*}
	U_{\delta} = \{f\in W^{2-1/p,p}(\partial\Omega,\mathbb{R}^{n})\colon \norm{f}_{W^{2-1/p,p}(\partial\Omega,\mathbb{R}^{n})}< \delta\}.
\end{equation*}
In view of Proposition \ref{lemma-solution-map}, there exists $\delta=\delta(\lambda)>0$ such that problem \eqref{eq-intro-quasilinear}, with  $f\in U_{\delta}$ and $g=\lambda+f$, admits a unique solution \( u_{\lambda,f}\in V_{\lambda} \). Then we can define the displacement-to-traction map as follows
\begin{equation}\label{eq-nonlinear-conormal}
	\mathcal{N}_{\lambda,C}\colon U_{\delta}\mapsto W^{1-1/p,p}(\partial\Omega,\mathbb{R}^{n}),\quad\mathcal{N}_{\lambda,C}(f) = [C(u_{\lambda,f},\epsilon(u_{\lambda,f}))\colon \epsilon(u_{\lambda,f})\nu]\vert_{\partial\Omega}.
\end{equation}
Here \( \nu\colon\partial\Omega\to\mathbb{R}^{n} \) is the outward pointing unit normal vector to $\partial\Omega$.
Note that we use solutions \( u\in W^{2,p}(\Omega,\mathbb{R}^{n}) \) with \( p>n \), so that the point measurements on the boundary are well-defined by the Sobolev embedding \( W^{1-1/p,p}(\partial\Omega,\mathbb{R}^{n})\hookrightarrow C(\partial\Omega,\mathbb{R}^{n}) \).

We mention that similar definitions of local Dirichlet-to-Neumann maps were introduced for semilinear elliptic equations in \cite{FO,LLYS}, for the stationary Navier-Stokes equations in \cite{lw07}, and for a nonlinear Lam\'e system in \cite{NaSu}.

The inverse problem under investigation in the present article can be stated as follows:
\begin{center}
\begin{minipage}{0.8\textwidth}
\begin{itemize}
	\item[{\bfseries (IP):}]\label{ip} Determine the nonlinear material parameter $C$ from knowledge of the point measurements
\begin{equation*}
	[\mathcal{N}_{\lambda,C}(tf_{B})](x),\quad x\in\mathcal{O},\ B\in\mathcal{M},\ \abs{t}<\tau,\ \lambda\in[-R,R]^{n},
\end{equation*}
where \( \mathcal{M}\subset S_n(\R) \), \( \mathcal{O}\subset \partial\Omega \) are suitable finite sets, \( f_{B}(x) = Bx \), \( \tau>0 \) is sufficiently small, and \( R>0 \) is arbitrary.
\end{itemize}
\end{minipage}
\end{center}

We separate our main results into two categories.
We consider first the determination of the dependence of the nonlinear material property \( C(\lambda,0) \) with respect to the displacement \( \lambda\in\R^n \) but with the strain fixed at \( \eta = 0 \).
Then, we extend these results into the identification of the full tensor \( C(\lambda,\eta) \) with dependence on the strain tensor \( \eta\in S_n(\R) \).
We establish both uniqueness results and a stability estimate.
\begin{remark}\label{r1}
We focus the presentation in the article on the constitutive equation \eqref{law} using the symmetric gradient \( \epsilon(u) \) and  the symmetries \ref{assumption-symmetry}, motivated by the modeling in continuum mechanics.
However, the method we develop can be extended, under minor modifications, to more general constitutive equations
\bel{stress11}\sigma(u)=C(u,\nabla u):\nabla u\ee
without assuming the symmetries \ref{assumption-symmetry}.
The proofs of the below theorems can easily be adapted to \eqref{stress11}, mostly by replacing the subspace \( S_{n}(\mathbb{R}) \) of symmetric matrices with \( \mathbb{R}^{n\times n} \) and by replacing the symmetric matrix \( B_{ij} = \frac{1}{2}(E_{ij}+E_{ji}) \) with \( E_{ij} \), where \( E_{ij} \) is the matrix whose only nonzero entry is at row \( i \) and column \( j \) and is equal to \( 1 \).
The results thus obtained can be found in Appendix \ref{appendix-0}.
\end{remark}

\subsection{Dependence with respect to the displacement}
We consider the recovery of \( C(\lambda,0) \), $\lambda\in\R^n$, from boundary measurements.
The tensor \( C(\lambda,\eta) \) may depend on the strain tensor \( \eta\in S_n(\R) \), but in this section we consider only the recovery of the dependence of the displacement \( \lambda\in\R^n \) when \( \eta = 0 \).

Denote by $E_{ij}$ the matrix with only one nonzero entry which is equal to 1 at index $ij$, i.e. $E_{ij}:=(\delta_{ik}\delta_{j\ell})_{1\leq k,\ell\leq n}$ where \( \delta_{ij} \) denotes the Kronecker delta and $i,j=1,\ldots,n$.
Let \( \mathcal{W} \) be the collection of matrices obtained by symmetrizing \( E_{ij} \),
\begin{equation}\label{eq-matrix-one-nonzero}
	\mathcal{W} = \{B\in S_{n}(\mathbb{R}): B = \frac{1}{2}(E_{ij} + E_{ji}) \text{ for some }i,j = 1,\ldots, n\}.
\end{equation}
Notice that Card$(\mathcal W)=\frac{n(n+1)}{2}$.

Let $\{e_1,\ldots,e_n\}$ be the canonical orthonormal basis of $\mathbb R^n$ and let the points \( x_{1},\ldots,x_{n}\in\partial\Omega \) be chosen so that, for all $i\in\{1,\ldots,n\}$, the normal vector \( \nu(x_{i}) \) is close to \( e_{i} \), as follows
\begin{equation}\label{eq-normal-vector-condition}
	\norm{\nu(x_{i})-e_{i}}_{\mathbb{R}^{n}}< \frac{1}{\sqrt n}.
\end{equation}
Denote this collection of points by \( \mathcal{V}=\{x_{1},\ldots,x_{n}\} \). From the surjectivity of the map $\nu\colon\partial\Omega\to\mathbb S^{n-1}$ (see e.g. \cite[Proposition A.1.]{kian24}), where $\mathbb S^{n-1}:=\{y\in\mathbb R^n:\ \|y\|_{\mathbb R^n}=1\}$, we know that the condition \eqref{eq-normal-vector-condition} will be fulfilled for some large collection of points \( x_{1},\ldots,x_{n}\in\partial\Omega \).

Our first main result concerns the problem \hyperref[ip]{(IP)} with $\mathcal M=\mathcal{W}$ and $\mathcal O=\mathcal{V}$. Namely, we consider the unique identification of $C(\lambda,0)$, $\lambda\in\mathbb R^n$  as follows.

\begin{theorem}\label{theorem-first-linearization}
Let $\lambda\in\R^n$,  \( \Omega \) satisfy assumption \ref{assumption-domain} and, for $m=1,2$, let \( C^m \) satisfy assumptions \ref{assumption-nonlinearity}, with $k=1$,  \ref{assumption-ellipticity} and \ref{assumption-symmetry}.
If for some \( \delta>0 \) it holds that
\begin{equation}\label{eq-first-linearization-1}
	[\mathcal{N}_{\lambda,C^{1}}(tf_{B})](x) = [\mathcal{N}_{\lambda,C^{2}}(tf_{B})](x)\quad  B\in \mathcal{W},\ x\in\mathcal{V},\  \abs{t}<\delta,
\end{equation}
 then $C^1(\lambda,0) = C^2(\lambda,0)$.
\end{theorem}
\begin{remark}
By definition of the displacement-to-traction maps \( \mathcal{N}_{\lambda,C^{1}},\mathcal{N}_{\lambda,C^{2}} \), we have \( \mathcal{N}_{\lambda,C^{1}}(0) = \mathcal{N}_{\lambda,C^{2}}(0) = 0 \). In other words, \( 0 \) belongs to the domain of both maps, so that \( \operatorname{Dom}(\mathcal{N}_{\lambda,C^{1}})\cap \operatorname{Dom}(\mathcal{N}_{\lambda,C^{2}}) \) is a nonempty open set.
The constant \( \delta \) in Theorem \ref{theorem-first-linearization} and the following theorems is implicitly chosen to that \( tf_{B}\in\operatorname{Dom}(\mathcal{N}_{\lambda,C^{1}})\cap\operatorname{Dom}(\mathcal{N}_{\lambda,C^{2}}) \) whenever \( \abs{t}< \delta \).
Both \( \mathcal{N}_{\lambda,C^{1}}(tf_{B}) \) and \( \mathcal{N}_{\lambda,C^{2}}(tf_{B}) \) are therefore well-defined.
\end{remark}

This result can be  improved for more specific class of nonlinear Lam\'e systems. For instance, let us consider nonlinear isotropic and homogeneous elastic medium where the stiffness tensor $C(u,\epsilon(u))$ is characterized by
\bel{iso}C(\lambda,\eta)=(\Lambda(\lambda,\eta)\delta_{ij}\delta_{k\ell}+\mu(\lambda,\eta)(\delta_{ik }\delta_{j\ell}+\delta_{i\ell }\delta_{jk}))_{1\leq i,j,k,\ell\leq n},\  (\lambda,\eta)\in\mathbb R^n\times S_n(\mathbb R).\ee
We obtain the determination of such class of nonlinear terms  from measurements at a single point (i.e. $\operatorname{Card}(\mathcal O)=1$).

\begin{corollary}\label{c1}
Suppose the assumptions of Theorem \ref{theorem-first-linearization} holds.
Let $x_0\in\partial\Omega$ be a point where the normal vector \( \nu(x) \) has at least two nonzero components and let \( B = B_{kk} \) be fixed.
Assume additionally that the tensors \( C^{m} \) are of the form
\begin{equation*}
C_{ijk\ell}^{m}(\lambda,\eta)=\Lambda^{m}(\lambda,\eta)\delta_{ij}\delta_{k\ell}+\mu^{m}(\lambda,\eta)(\delta_{ik }\delta_{j\ell}+\delta_{i\ell }\delta_{jk}),\quad  (\lambda,\eta)\in\mathbb R^n\times S_n(\mathbb R),
\end{equation*}
where $\Lambda^m\in C^2(\mathbb R^n\times S_n(\mathbb R),\mathbb R) $ and $\mu^m\in C^2(\mathbb R^n\times S_n(\mathbb R),\mathbb R)$.
If  for some $\lambda\in\R^n$ and \( \delta>0 \) it holds that
\begin{equation}\label{c1a}
	[\mathcal{N}_{\lambda,C^{1}}(tf_{B})](x_{0}) = [\mathcal{N}_{\lambda,C^{2}}(tf_{B})](x_{0}),\quad\abs{t}<\delta,
\end{equation}
then $C^1(\lambda,0) = C^2(\lambda,0)$.

\end{corollary}
To the best of our knowledge, we give in Theorem \ref{theorem-first-linearization} the first positive answer to problem \hyperref[ip]{(IP)} for the dependence of the nonlinear material property with respect to the displacement.
This extends to Lam\'e systems earlier results for scalar equations with space-independent nonlinearities \cite{kian24,MU}.
Regarding space-dependent tensors, we are aware only of the work \cite{NaSu}, which establishes a local uniqueness result by reduction to the corresponding results on linear isotropic elasicity in \cites{NU94}.

We prove in Corollary \ref{c1} that, for isotropic nonlinear elastic tensors of the form \eqref{iso},  the result of Theorem \ref{theorem-first-linearization} is still true from   measurements at a single point prescribed by excitations on a single direction $f_B$. More precisely, we consider \hyperref[ip]{(IP)}
with $\operatorname{Card}(\mathcal O)=\operatorname{Card}(\mathcal M)=1$ and we recover the two distinct Lam\'e moduli $\Lambda$ and $\mu$.

The uniqueness result of Theorem \ref{theorem-first-linearization} can be extended to the following stability estimate.
\begin{theorem}\label{theorem-first-linearization-stability}
Suppose the assumptions of Theorem \ref{theorem-first-linearization} holds, but now let \( \lambda \) belong to some \( n \)-cube \( \mathcal{Q} = [-R,R]^{n}\subset\mathbb{R}^{n} \) for \( R>0 \) arbitrary.
Then
\begin{equation}\label{est1}
	\sup_{\lambda\in \mathcal{Q}}\norm{ C^1(\lambda,0)-C^2(\lambda,0)}_{*} \leq C\sup_{\substack{B\in\mathcal{W}\\ \lambda\in\mathcal{Q}\\ x\in\mathcal{V}}}\norm{[D\mathcal{N}_{\lambda,C^{1}}(0)f_{B}](x)-[D\mathcal{N}_{\lambda,C^{2}}(0)f_{B}](x)}_{\mathbb{R}^{n}},
\end{equation}
where \( C \) is given by
\begin{equation*}
	C = \frac{\frac{n(n+1)}{2}\sqrt{2n(1+nM^{2})}}{1-nM^{2}}, \quad M = \max_{i=1,\ldots,n}\{\norm{\nu(x_{i})-e_{i}}_{\mathbb{R}^{n}}\}.
\end{equation*}
Above, \( \norm{\cdot}_{*} \) denotes the operator norm of fourth order tensors viewed as linear transformations on matrices.
\end{theorem}

Observe that the estimate \eqref{est1} is a Lipschitz stability estimate with an explicit constant completely independent of the unknown parameter.
In that sense, this stability estimate  is an unconditional stability.
This confirms what was already observed for scalar elliptic equations in \cite{Ki23,kian24}.
This stability estimate as well as the data used in Theorem \ref{theorem-first-linearization-stability} can be useful for different reconstruction algorithms based on iterative methods such as Tikhonov regularization (see, e.g. \cite{CY} for more details).

Similarly, we can improve Theorem \ref{theorem-first-linearization-stability} for nonlinear isotropic and homogeneous elastic medium with stiffness tensor $C(u,\epsilon(u))$ of the form 
\eqref{iso}.

\begin{corollary}\label{c2} 
Suppose the assumptions of Corollary \ref{c1} holds, but now let \( \lambda \) belong to some \( n \)-cube \( \mathcal{Q} = [-R,R]^{n}\subset\mathbb{R}^{n} \) for \( R>0 \) arbitrary.
Then
\begin{equation}\label{c2a}
\begin{aligned}
	\norm{C^{1}(\lambda,0)-C^{2}(\lambda,0)}_{*} &\leq \sup_{\lambda\in\mathcal Q}\abs{ \Lambda^1(\lambda,0)-\Lambda^2(\lambda,0)}+\sup_{\lambda\in\mathcal{Q}}\abs{ \mu^1(\lambda,0)-\mu^2(\lambda,0)} \\
	&\leq C\sup_{\lambda\in\mathcal{Q}}\norm{[D\mathcal{N}_{\lambda,C^{1}}(0)f_{B_{kk}}](x_{0}) - [D\mathcal{N}_{\lambda,C^{2}}(0)f_{B_{kk}}](x_{0})}_{\mathbb{R}^{n}},
\end{aligned}
\end{equation}
where \( C = (2|\nu_k(x_0)|^{-1}+ |\nu_l(x_0)|^{-1}) \) and \( \nu_{\ell}(x_{0)},\nu_{k}(x_{0}) \) are two nonzero components of the normal vector \( \nu(x_{0}) \).
\end{corollary}

\subsection{Dependence with respect to the strain tensor}
We consider here the \( \eta \)-dependence of the elastic tensor $C(\lambda,\eta)$, $\lambda\in\mathbb R^n$, $\eta\in S_n(\mathbb R)$.
Our first result concerns the unique recovery of the Taylor coefficients $D_\eta^kC(\lambda,0)$, $k\in\mathbb N$, $\lambda\in\mathbb R^n$ when $C(\lambda,\eta)=(C_{ijk\ell}(\lambda,\eta))_{1\leq i,j,k,\ell\leq n}$ takes the form
\bel{CC1}C_{ijk\ell}(\lambda,\eta)=\left\{\begin{array}{ll}\mu_{ij}(\lambda,\eta)D_{ijk\ell}(\lambda)&\text{ if }(i,j)= (k,\ell),\\
D_{ijk\ell}(\lambda)&\text{ if }(i,j)\neq (k,\ell),\end{array}\right.
\quad \lambda\in\mathbb R^n,\ \eta\in S_n(\mathbb R). \ee
We consider the subset $\mathcal W_\star$ of $S_n(\mathbb R)$ given by
\begin{equation*}
\mathcal W_\star=\mathcal{W}\cup \{B_{ij}+B_{k\ell}:\ i,j,k,\ell=1,\ldots,n,\ (i,j)\neq(k,\ell)\},
\end{equation*}
where \( \mathcal{W} \) is given by \eqref{eq-matrix-one-nonzero}.
For all $\ell\in\mathbb N$, we introduce also the set $\mathcal W_\star^\ell$ defined by
\begin{equation*}
	\mathcal W_\star^\ell=\{B_1+\ldots+B_k:\ k=1,\ldots,\ell,\ B_1,\ldots,B_k\in\mathcal W_\star\}.
\end{equation*}
Then, we consider problem \hyperref[ip]{(IP)} with $\mathcal M=\mathcal W_\star^\ell$, for some $\ell\in\mathbb N$, and $\mathcal O=\mathcal V$.

Our first main result about recovery of the \( \eta \)-dependence of the elastic tensor, can be stated as follow.
\begin{theorem}\label{t4}
Let \( \Omega \) satisfy assumption \ref{assumption-domain} and let $N\geq2$.
For $m=1,2$, let \( C^m \) be tensors satisfying \ref{assumption-nonlinearity}, with $k=N$, as well as assumptions \ref{assumption-ellipticity} and \ref{assumption-symmetry}.
Assume additionally that \( C^{m} \) take the form
\bel{t4a}
	C^m_{ijk\ell}(\lambda,\eta)=
	\begin{cases}
		\mu_{ij}^m(\lambda,\eta)D^m_{ijk\ell}(\lambda)&\text{if }(i,j)= (k,\ell),\\
		D^m_{ijk\ell}(\lambda)&\text{if }(i,j)\neq (k,\ell),
	\end{cases}
\ee
where \( (\lambda,\eta)\in \mathbb{R}^{n}\times S_{n}(\mathbb{R}) \) and
\begin{equation*}
\begin{aligned}
	\mu^{m} &= (\mu_{ij}^m)_{1\leq i,j\leq n}\in C^{N+1}(\mathbb R^n\times S_n(\mathbb R), \mathbb R^{n\times n}), \\
	D^{m} &= (D_{ijk\ell}^m)_{ 1\leq i,j,k,\ell\leq n}\in C^{N+1}(\mathbb R^n,\mathbb R^{n\times n\times n\times n}),
\end{aligned}
\end{equation*}
and \( \mu_{ij}^m(\lambda,0)=1 \) for all $i,j=1,\ldots,n$. 
If, for $R>0$ and \( \delta>0 \), it holds that
\begin{equation}\label{t4ab}
	[\mathcal{N}_{\lambda,C^{1}}(tf_{B})](x) = [\mathcal{N}_{\lambda,C^{2}}(tf_{B})](x),\quad B\in \mathcal{W}_{\star}^{N-1},\ x\in\mathcal{V},\  \abs{t}<\delta,\ \lambda\in[-R,R]^{n},
\end{equation}
 then
\begin{equation}\label{t4c}
	D_{\eta}^kC^1(\lambda,0) = D_{\eta}^kC^2(\lambda,0)\qquad k=0,1,\ldots,N-1,\ \lambda\in[-R,R]^n.
\end{equation}
\end{theorem}

As a straightforward application of this result, we prove the following uniqueness result.
\begin{corollary}\label{c4}
	Let the conditions of Theorem \ref{t4} be fulfilled.
	Assume additionally that, for $m=1,2$, the map $S_n(\mathbb R)\ni\eta\mapsto C^{m}(\lambda,\eta)\in \mathbb{R}^{n\times n\times n\times n}$ is real-analytic for each fixed \( \lambda\in[-R,R]^{n} \).
	If for some \( \delta>0 \), \( R>0 \) it holds that
	\begin{equation*}
		[\mathcal{N}_{\lambda,C^{1}}(tf_{B})](x) = [\mathcal{N}_{\lambda,C^{2}}(tf_{B})](x),\quad B\in \bigcup_{\ell\geq1}\mathcal{W}_{\star}^\ell,\ x\in\mathcal{V},\  \abs{t}<\delta,\ \lambda\in[-R,R]^{n},
	\end{equation*}
	 then
	\begin{equation*}
		C^1(\lambda,\eta)=C^2(\lambda,\eta),\quad (\lambda,\eta)\in [-R,R]^n\times S_n(\mathbb R).
	\end{equation*}
\end{corollary}
For elastic tensors $C(\lambda,\eta)=(C_{ijk\ell}(\lambda,\eta))_{1\leq i,j,k,\ell\leq n}$ of the form
\bel{CC2}C_{ijk\ell}(\lambda,\eta)=\mu_{ij}(\lambda,\eta)D_{ijk\ell}(\lambda),\quad \lambda\in\mathbb R^n,\ \eta\in S_n(\mathbb R), \ee
we can get a result similar to Theorem \ref{t4} by considering problem \hyperref[ip]{(IP)} with boundary measurements generated by the set \( \mathcal{M} =\mathcal{W}_{\diamond}^{\ell} \) defined as
\begin{equation*}
	\mathcal W_{\diamond}^{\ell}=\{B_1+\ldots+B_k:\ k=1,\ldots,\ell,\ B_1,\ldots,B_k\in\mathcal W_{\diamond}\},
\end{equation*}
where
\begin{equation}\label{eq-measurement-matrices-1}
	\mathcal W_{\diamond} = \{\pm B:\ B\in \mathcal W\}\cup \{B_{ij}+B_{k\ell}:\ i,j,k,\ell=1,\ldots,n,\ (i,j)\neq(k,\ell)\},
\end{equation}
and \( \mathcal{W} \) is given by \eqref{eq-matrix-one-nonzero}.
\begin{theorem}\label{t5}
	Let the conditions of Theorem \ref{t4} be fulfilled, except now let the tensors be given by
\bel{t5a}
	C^m_{ijk\ell}(\lambda,\eta)=\mu_{ij}^m(\lambda,\eta)D^m_{ijk\ell}(\lambda),\quad \lambda\in\mathbb R^n,\ \eta\in S_n(\mathbb R),\ m=1,2.
\ee
If, for $R>0$ and \( \delta>0 \), it holds that
\begin{equation}\label{t5ab}
	[\mathcal{N}_{\lambda,C^{1}}(tf_{B})](x) = [\mathcal{N}_{\lambda,C^{2}}(tf_{B})](x),\quad  B\in \mathcal{W}_{\diamond}^{N-1},\ x\in\mathcal{V},\  \abs{t}<\delta,\ \lambda\in[-R,R]^{n},
\end{equation}
 then
\begin{equation}\label{t5c}
	D_{\eta}^kC^1(\lambda,0) = D_{\eta}^kC^2(\lambda,0)\qquad k=0,1,\ldots,N-1,\ \lambda\in[-R,R]^n.
\end{equation}
\end{theorem}

We end this section with a uniqueness result using measurements in a single point on \( \partial\Omega \) for isotropic elastic tensor of the form \eqref{iso}.
For this purpose, let us fix $k,\ell\in\{1,\ldots,n\}$, $k\neq\ell$, and consider for \(m\in\mathbb{N} \) the subset \(\mathcal{W}_{k,\ell}^m\subset S_n(\mathbb R)\) given by
\begin{equation}\label{isotropic-measurement-matrices}
	\mathcal{W}_{k,\ell}^m = \mathcal{W}_{k,\ell}^{+,m}\cup \mathcal{W}_{k,\ell}^{-,m},
\end{equation}
where
\begin{equation*}
\begin{aligned}
	\mathcal{W}_{k,\ell}^{\pm,m} &=\{A_1+\ldots A_q:\ q=1,\ldots,m,\ A_j\in \mathcal{W}_{k,\ell}^{\pm},\ j=1,\ldots,q\}, \\
	\mathcal{W}_{k,\ell}^{+} &= \{B_{kk}+B_{ij}: (i,j)\neq (k,k),(\ell,\ell)\}\cup \{B_{kk}, 2B_{kk}-B_{\ell\ell}\}, \\
	\mathcal{W}_{k,\ell}^{-} &= \{B_{\ell\ell} - B_{ij}:(i,j)\neq(k,k),(\ell,\ell)\}\cup \{B_{\ell\ell}, 2B_{\ell\ell}-B_{kk}\}.
\end{aligned}
\end{equation*}
We can extend the result of Theorem \ref{t4} to an isotropic elastic tensor of the form \eqref{iso} as follows.

\begin{theorem}\label{c3} 
	Let the conditions of Theorem \ref{t4} be fulfilled, except now let the tensors be given by
\begin{equation}\label{eq-corollary-3-1}
	C^{m}(\lambda,\eta)=(\Lambda^{m}(\lambda,\eta)\delta_{ij}\delta_{k\ell}+\mu^{m}(\lambda,\eta)(\delta_{ik }\delta_{j\ell}+\delta_{i\ell }\delta_{jk}))_{1\leq i,j,k,\ell\leq n}, m=1,2,
\end{equation}
with \( \Lambda^m\in C^{N+1}(\mathbb R^n\times S_n(\mathbb R),\mathbb R) \) and \( \mu^m\in C^{N+1}(\mathbb R^n\times S_n(\mathbb R),\mathbb R) \).
Let $x_0\in\partial\Omega$ be a point where the normal vector \( \nu(x_{0}) \) has at least two nonzero components, say \( \nu_{k}(x_{0}) \) and \( \nu_{\ell}(x_{0}) \), such that
\begin{equation}\label{c3a}
\begin{aligned}
	&\abs{\nu_{k}(x_{0})}\neq \abs{\nu_{\ell}(x_{0})} \\
	&|\nu_{k}(x_{0})|,|\nu_{\ell}(x_{0})| \geq |\nu_{j}(x_{0})|,\quad j\neq k,\ell.
\end{aligned}
\end{equation}
If for some \( \delta>0 \), \( R>0 \) it holds that
\begin{equation}\label{c3bc}
	[\mathcal{N}_{\lambda,C^{1}}(tf_{B})](x_{0}) = [\mathcal{N}_{\lambda,C^{2}}(tf_{B})](x_{0})\quad\forall B\in \mathcal{W}_{k,\ell}^{N-1},\abs{t}<\delta,\ \lambda\in[-R,R]^{n},
\end{equation}
 then
\begin{equation}\label{c3b}	D_{\eta}^k\Lambda^1(\lambda,0) = D_{\eta}^k\Lambda^2(\lambda,0), \quad D_{\eta}^k\mu^1(\lambda,0) = D_{\eta}^k\mu^2(\lambda,0),
\end{equation}
for all \( k=0,1,\ldots,N-1,\ \lambda\in[-R,R]^n \).
\end{theorem}

 In Theorems \ref{t4} and \ref{t5}, we solve the inverse problem \hyperref[ip]{(IP)} by characterizing the dependence of nonlinear material parameters on both the displacement and the strain tensor.
Under suitable real-analyticity assumptions, these results yield the complete determination of elastic tensors \( C(\lambda,\eta) \) of the forms \eqref{CC1} and \eqref{CC2}.
To the best of our knowledge, this provides the first such result for the Lam\'e system with nonlinearities depending simultaneously on the displacement and the strain tensor.

The analysis extends earlier works on scalar elliptic equations \cite{kian24,MU} to Lam\'e systems under restricted boundary data.
Specifically, we consider stress boundary measurements generated by boundary displacements of the form $\lambda + t f_B$, with $\lambda \in \mathbb{R}^n$ and $B$ ranging over finite sets \( \mathcal M = \mathcal{W}_{\star}^{\ell} \) and \( \mathcal M = \mathcal{W}_{\diamond}^{\ell} \), $\ell\in\mathbb N$,  respectively.
Even in the scalar elliptic setting, the use of such restricted data for solving \hyperref[ip]{(IP)} appears to be new.

The proofs rely on explicit asymptotic expansions of the boundary stress measurements $\mathcal N_{\lambda,C}(tf_B)$ as $t \to 0$, obtained by combining Taylor expansions with a linearization of the boundary stress operator.
The approach is inspired by \cite{Is1,KLU,SU97} and extends infinite-order linearization techniques \cite{KLU} to stationary Lam\'e systems.
In contrast to the scalar elliptic case, the Lam\'e system requires a tensorial formulation of the linearization procedure and its implementation in Sobolev spaces.

In Theorem \ref{c3}, we establish the full identification of isotropic nonlinear elastic tensors of the form \eqref{iso} from measurements at a single point, including the simultaneous determination of the Lam\'e moduli $\Lambda(u,\epsilon( u))$ and $\mu(u,\epsilon( u))$, despite the finiteness of the excitation space.
The technical difficulties arising from these restrictions are resolved by exploiting the tensorial structure and asymptotic properties of the boundary stress operator, together with algebraic properties of matrix-valued multilinear forms and unique solvability of an equation involving the Lam\'e moduli. The results of Theorem \ref{c3} depend on the choice of boundary point measurement \( x_{0} \), and we derive unique recovery  under the rather mild conditions \eqref{c3a}.

Although the uniqueness results of Theorems \ref{t4} and \ref{t5} may be strengthened to H\"older-type stability estimates under additional assumptions, we do not address such estimates here in order to avoid unnecessary technical complications.

\subsection{Outline}
This article is organized as follows. In Section \ref{sec3} we study the forward problem \eqref{eq-intro-quasilinear} and we introduce the definition of solutions close to a constant vector that will be consider in our analysis. Section \ref{sec4} will be devoted to the results related to the dependency with respect to the displacement of the nonlinear material parameters, while in Section \ref{sec5} we consider the dependence with respect to the stain tensor. In the Appendix we give the statement of our results when the relation  between the stress tensor and the strain tensor  given by \eqref{stress11} and we recall some properties of solutions of linear Lam\'e systems that will be exploited in Section \ref{sec3}.

\section{Forward problem}\label{sec3}
We will apply the implicit function theorem to solve \eqref{eq-intro-quasilinear},
and this relies on the well-posedness of the linearized equation.
The required solvability result for the linearized equation can be found in Appendix \ref{appendix-1}.

Consider \eqref{eq-intro-quasilinear}, with  $\lambda\in\R^n$ a constant vector and $g=\lambda+f$. If we formally linearize the nonlinear equation in \eqref{eq-intro-quasilinear} at the constant solution \( u = \lambda \)  then we obtain the constant coefficient linear system
\begin{equation}\label{eq-linearized}
\begin{cases}
	\operatorname{div}(C(\lambda,0):\varepsilon (v)) = 0&\text{in }\Omega,\\
	v = f&\text{on }\partial\Omega,
\end{cases}
\end{equation}
where \( C(\lambda,0)\in\mathbb{R}^{n\times n\times n\times n} \).
By using the well-posedness of this linear system we prove, using the implicit function theorem, local well-posedness for the nonlinear system \eqref{eq-intro-quasilinear}.
But first, we establish in the next lemma the required mapping properties of the map \( u\mapsto C(u,\nabla u) \).
\begin{lemma}\label{lemma-nonlinearity-smoothness}
Let \( \Omega \) satisfy the assumption \ref{assumption-domain} and  let \( f\in C^{k+1}(\mathbb{R}^{n}\times\mathbb{R}^{n\times n},\mathbb{R}) \).
If \( p>n \) and \( u\in W^{2,p}(\Omega,\mathbb{R}^{n}) \) then \( u\mapsto f(u,\nabla u) \eqqcolon G(u)\) defines a \( C^{k} \) map from \( W^{2,p}(\Omega,\mathbb{R}^{n})\)
to \(W^{1,p}(\Omega) \).
\end{lemma}
\begin{proof}
For \( z\in\mathbb{R}^{n},A\in\mathbb{R}^{n\times n} \) we denote the partial derivatives of \( f(z,A) \) by \( \partial_{z_{i}}f \) and \( \partial_{A_{ij}}f \).
We start by considering the case $k=1$.
Note that since \( p>n \), each partial derivative \( \partial_{x_{j}}u_{i} \) belongs to \( C(\overline{\Omega}) \) by the Sobolev embedding.
Hence the function
\begin{equation*}
	\Omega\ni x\mapsto f(u(x),\nabla u(x))\in\mathbb{R}
\end{equation*}
is uniformly continuous and therefore belongs to \( L^{p}(\Omega) \).
Let
\begin{equation*}
	\psi_{m}\in C^{\infty}(\Omega,\mathbb{R}^{n})\cap W^{2,p}(\Omega,\mathbb{R}^{n})
\end{equation*}
be a sequence converging to \( u \) in \( W^{2,p}(\Omega,\mathbb{R}^{n}) \).
Then we have for any test function \( \varphi\in C_{c}^{\infty}(\Omega) \) that
\begin{equation*}
\begin{aligned}
	\int_{\Omega}f(u,\nabla u)\partial_{x_{\ell}}\varphi\,dx = &\lim_{m\to\infty} \int_{\Omega}f(\psi_{m},\nabla \psi_{m})\partial_{x_{\ell}}\varphi\,dx \\
	= &-\lim_{m\to\infty} \int_{\Omega}\partial_{x_{\ell}}\Big(f(\psi_{m},\nabla \psi_{m})\Big)\varphi\,dx \\
	= &-\lim_{m\to\infty} \int_{\Omega}\sum_{i=1}^{n}\partial_{z_{i}}f(\psi_{m},\nabla \psi_{m})\partial_{x_{\ell}}\psi_{i,m}\varphi\,dx \\
	&-\lim_{m\to\infty} \int_{\Omega}\sum_{i,j=1}^{n}\partial_{A_{ij}}f(\psi_{m},\nabla \psi_{m})\partial^{2}_{x_{\ell}x_{j}}\psi_{i,m}\varphi\,dx \\
	= &-\int_{\Omega}\sum_{i=1}^{n}\partial_{z_{i}}f(u,\nabla u)\partial_{x_{\ell}}u_{i}\varphi\,dx \\
	&-\int_{\Omega}\sum_{i,j=1}^{n}\partial_{A_{ij}}f(u,\nabla u)\partial^{2}_{x_{\ell}x_{j}}u_{i}\varphi\,dx.
\end{aligned}
\end{equation*}
This shows that the \( L^{p}(\Omega) \) function
\begin{equation*}
	\sum_{i=1}^{n}\partial_{z_{i}}f(u,\nabla u)\partial_{x_{\ell}}u_{i}+\sum_{i,j=1}^{n}\partial_{A_{ij}}f(u,\nabla u)\partial^{2}_{x_{\ell}x_{j}}u_{i}
\end{equation*}
is a weak derivative of \( f(u,\nabla u) \).
Hence \( G(u) = f(u,\nabla u)\in W^{1,p}(\Omega) \) and the map is well-defined.

The Fréchet derivative of \( u\mapsto f(u,\nabla u) = G(u) \) is given by
\begin{equation*}
	G'(u)v = \sum_{i}\partial_{z_{i}}f(u,\nabla u)v_{i}+\sum_{ij}\partial_{A_{ij}}f(u,\nabla u)\partial_{x_{j}}v_{i},
\end{equation*}
which follows from Taylor's formula
\begin{equation*}
\begin{aligned}
	f(z+y,A+B)&-f(z,A) -\sum_{i}\partial_{z_{i}}f(z,A)y_{i}-\sum_{ij}\partial_{A_{ij}}f(z,A)B_{ij} \\
	&= \int_{0}^{1}\sum_{i}\Big[\partial_{z_{i}}f(z+ty,A+tB)-\partial_{z_{i}}f(z,A)\Big]y_{i}\,dt \\
	&+\int_{0}^{1}\sum_{ij}\Big[\partial_{A_{ij}}f(z+ty,A+tB)-\partial_{A_{ij}}f(z,A)\Big]B_{ij}\,dt.
\end{aligned}
\end{equation*}
Since the derivatives of \( f \) are locally Lipschitz, we have for any compact set \( K\subset\mathbb{R}^{n}\times \mathbb{R}^{n\times n} \) a constant \( C \) such that for \( (z,A),(y,B)\in K \) the follows inequalities hold
\begin{equation*}
\begin{aligned}
	\lvert f(z+y,A+B)&-f(z,A) -\sum_{i}\partial_{z_{i}}f(z,A)y_{i}-\sum_{ij}\partial_{A_{ij}}f(z,A)B_{ij}\rvert \\
	&\leq \frac{C}{2}\sum_{i}(\norm{y}_{\mathbb{R}^{n}} + \norm{B}_{\mathbb{R}^{n\times n}})\abs{y_{i}} \\
		&+\frac{C}{2}\sum_{ij}(\norm{y}_{\mathbb{R}^{n}}+\norm{B}_{\mathbb{R}^{n\times n}})\abs{B_{ij}} \\
		&\leq C (\norm{y}_{\mathbb{R}^{n}}^{2} + \norm{B}_{\mathbb{R}^{n\times n}}^{2}).
\end{aligned}
\end{equation*}
If \( u,v\in W^{2,p}(\Omega,\mathbb{R}^{n}) \) then the ranges of \( (u,\nabla u) \) and \( (v,\nabla v) \) are compact sets and if \( \norm{v}_{W^{2,p}(\Omega),\mathbb{R}^{n}}\leq \delta \) for some fixed \( \delta>0 \) then the range of the perturbation \( (u+tv,\nabla(u+tv)) \) belong to a fixed compact set.
This allow us to use the above estimate to prove
\begin{equation*}
\begin{aligned}
	\norm{G(u+v)-G(u)-G'(u)v}_{L^{p}(\Omega)} \leq C\norm{v}_{W^{2,p}(\Omega,\mathbb{R}^{n})}^{2}.
\end{aligned}
\end{equation*}
Hence \( G \) is differentiable as a map \( W^{2,p}(\Omega,\mathbb{R}^{n})\to L^{p}(\Omega,\mathbb{R}^{n}) \)
In a similar but somewhat more tedious way we obtain the same estimate for the partial derivatives
\begin{equation*}
\begin{aligned}
	\norm{\partial_{x_{j}}G(u+v)-\partial_{x_{j}}G(u)-\partial_{x_{j}}G'(u)v}_{L^{p}(\Omega)} \leq C\norm{v}_{W^{2,p}(\Omega,\mathbb{R}^{n})}^{2},
\end{aligned}
\end{equation*}
and it follows that \( G \) is differentiable as a map \( W^{2,p}(\Omega,\mathbb{R}^{n})\to W^{1,p}(\Omega,\mathbb{R}^{n}) \).
Continuity of the derivative is established similarly, so that \( G \) is \( C^{1} \).

Higher regularity is established by iterating the above argument as follows.
Define the two new functions
\begin{equation*}
	\tilde{f}_{y,B}(z,A) = \sum_{i}\partial_{z_{i}}f(z,A)y_{i}+\sum_{i,j}\partial_{A_{ij}}f(z,A)B_{ij}
\end{equation*}
and
\begin{equation*}
	\tilde{G}_{v}(u) = \tilde{f}_{v,\nabla v}(u,\nabla u) = G'(u)v.
\end{equation*}
Since \( f \) is \( C^{k+1} \) then \( \tilde{f}_{y,B} \) is \( C^{k} \).
By replacing \( f,G \) in the above argument with \( \tilde{f}_{y,B},\tilde{G}_{v} \) we conclude that \( \tilde{G}_{v} \) is \( C^{1} \).
But this implies that \( G \) is \( C^{2} \), since \( \tilde{G}_{v}^{\prime}(u)w = G^{\prime\prime}(u)vw \).
Repeating this we conclude that \( G \) is \( C^{k} \).
\end{proof}

\begin{proposition}\label{lemma-solution-map}
Let \( \Omega \) satisfy assumption \ref{assumption-domain} and let \( C \) satisfy assumptions \ref{assumption-nonlinearity}, \ref{assumption-ellipticity} and \ref{assumption-symmetry}.
Let \( p>n \).
For every vector \( \lambda\in \mathbb{R}^{n} \) there exists a neighborhood \( U_\lambda\subset W^{2-1/p,p}(\partial\Omega,\mathbb{R}^{n}) \) around \( 0 \), a neighborhood \( V_\lambda\subset W^{2,p}(\Omega,\mathbb{R}^{n}) \) around \( \lambda \), and a \( C^{k} \) map \( S_\lambda\colon U_\lambda\to V_\lambda \) such that \( S_\lambda(f) = u \) is the unique strong solution in \( V_\lambda \) of
\begin{equation}\label{eq-solution-map-1}
\begin{cases}
	\operatorname{div}(C(u,\epsilon(u)):\epsilon(u)) = 0&\text{in }\Omega,\\
	u = \lambda+f&\text{on }\partial\Omega.
\end{cases}
\end{equation}

\end{proposition}
\begin{proof}
First note that Lemma \ref{lemma-nonlinearity-smoothness} shows that each component \( C_{ijk\ell}(u,\epsilon(u)) \) belongs to \( W^{1,p}(\Omega) \).
Hence \( C(u,\epsilon(u)):\epsilon(u) \) is a matrix whose entries belong to \( W^{1,p}(\Omega) \), since the space is an algebra when \( p>n \).
It therefore makes sense to take apply the divergence operator row-wise on the matrix.

When \( f = 0 \) then \( u = \lambda \) is a solution for any \( \lambda\in\mathbb{R}^{n} \).
Next, let \( \lambda \) be fixed and consider the map
\begin{equation*}
	Q(u,f) = \Big(\operatorname{div}(C(u,\epsilon(u)):\epsilon(u)), u\vert_{\partial\Omega}-\lambda-f\Big).
\end{equation*}
Lemma \ref{lemma-nonlinearity-smoothness} implies that \( Q \) is \( C^{k} \) as a map between
\begin{equation*}
	W^{2,p}(\Omega,\mathbb{R}^{n})\times W^{2-1/p,p}(\partial\Omega,\mathbb{R}^{n})\to L^{p}(\Omega,\mathbb{R}^{n})\times W^{2-1/p,p}(\partial\Omega,\mathbb{R}^{n}).
\end{equation*}
We have \( Q(\lambda, 0) = (0,0) \) and the \( u \)-derivative at \( (\lambda,0) \) in direction \( v \) is
\begin{equation*}
	D_{u}Q(\lambda,0)v = \Big(\operatorname{div}(C(\lambda,0):\epsilon(v)), v\Big).
\end{equation*}
By Corollary \ref{corollary-linear-well-posedness-symmetric} the linear equation is well-posed and hence \( D_{u}Q(\lambda,0) \) is a norm isomorphism.
The implicit function theorem \cite[Theorem 4]{hg27} ensures the existence of a neighborhood \( U_\lambda\subset W^{2-1/p,p}(\partial\Omega,\mathbb{R}^{n}) \) around \( 0 \), a neighborhood \( V_\lambda\subset W^{2,p}(\Omega,\mathbb{R}^{n}) \) around \( \lambda \), and a \( C^{k} \) map \( S_{\lambda}\colon U_\lambda\to V_\lambda \) such that \( S_{\lambda}(f) = u \) is the unique solution of \eqref{eq-solution-map-1} in \( \mathcal{V} \).
\end{proof}

In the remainder this article, we use the term \emph{unique strong solutions close to} $\lambda$ to mean the unique solution $u = S_{\lambda}(f)$ provided by Proposition \ref{lemma-solution-map}.

\section{Recovery of dependence on the displacement}\label{sec4}

Using Proposition \ref{lemma-solution-map}, for every $\lambda\in\R^n$, the displacement-to-traction map \( \mathcal{N}_{\lambda,C} \) is well-defined as a map \( W^{2-1/p,p}(\partial\Omega,\mathbb{R}^{n})\supset U_\lambda\ni f\mapsto \mathcal{N}_{\lambda,C}(f)\in W^{1-1/p,p}(\partial\Omega,\mathbb{R}^{n})\subset C(\partial\Omega,\R^n) \), for some neighborhood \( U_\lambda \) of zero.
Importantly, \( \mathcal{N}_{\lambda,C} \) is \( C^{1} \) and its Fréchet derivative \( f\mapsto D\mathcal{N}_{\lambda,C}(0)f \) is a similar boundary measurement map but for the linearized equation.
In the next couple of lemmas, we show that point measurements \( [D\mathcal{N}_{\lambda,C}(0)f_{B}](x) \) uniquely determines \( C(\lambda,0) \), whenever  \( B\in\mathcal{M} \) and \( x\in\mathcal{O} \).
Combining this with the fact that \( D\mathcal{N}_{\lambda,C}(0)f_{B} \) is determined by \( \mathcal{N}_{\lambda,C}(tf_{B}) \) for \( t \) in a neighborhood of zero, we obtain Theorem \ref{theorem-first-linearization} and Corollary \ref{c1}.
Our analysis leads to the stability estimates in Theorem \ref{theorem-first-linearization-stability} and Corollary \ref{c2}.

All these results rely on the linearized boundary measurements \( D\mathcal{N}_{\lambda,C}(0) \), which are related to the following linear system
\begin{equation}\label{eq1}
\begin{cases}
	\operatorname{div}(C(\lambda,0):\varepsilon(v)) = 0&\text{in }\Omega,\\
	v = f_{B}&\text{on }\partial\Omega.
\end{cases}
\end{equation}
By considering the solution of \eqref{eq1}, we obtain the following property about the first order linearization of solutions close to $\lambda$ for the nonlinear equation \eqref{eq-solution-map-1}.

\begin{lemma}\label{lemma-solution-map-conormal}
Let \( \Omega \) satisfy assumption \ref{assumption-domain} and let \( C \) satisfy assumptions \ref{assumption-nonlinearity}, with $k=1$,  \ref{assumption-ellipticity} and \ref{assumption-symmetry}.
Let \( p>n \) and \( \lambda\in\mathbb{R}^{n} \) and denote by \( \mathcal{N}_{\lambda,C} \) the local displacement-to-traction map for the equation \eqref{eq-solution-map-1}.
Let \( B\in S(\R^n) \) and let \( f_{B}(x)=Bx \).
Then
\begin{equation}\label{eq-solution-map-conormal-2}
	D\mathcal{N}_{\lambda,C}(0)f_{B} = [C(\lambda,0)\colon B]\nu.
\end{equation}
\end{lemma}
\begin{proof}
Let \( S_{\lambda} \) denote the solution map from Proposition \ref{lemma-solution-map} used in the definition of \( \mathcal{N}_{\lambda,C} \).
It follows from the trace theorem and Lemma \ref{lemma-nonlinearity-smoothness} that the boundary stress \( \mathcal{N}_{\lambda,C}(f) \) of \( S_\lambda(f) \) is \( C^{1} \) with respect to the Dirichlet data \( f\in \mathcal U \).
To find the Fréchet derivative of \( \mathcal{N}_{\lambda,C} \), first note that the function \( x\mapsto Bx \) is a solution of \eqref{eq1}.
Since \( v_{\lambda,B} \coloneqq S_{\lambda}'(0)f_{B} \) is also a solution, we conclude by uniqueness of solutions (see e.g. Corollary \ref{corollary-linear-well-posedness-symmetric}) that \( v_{\lambda,B}(x) = Bx \) for all \( x\in \Omega \).
Denote by \( u_{\lambda,f} = S_{\lambda}(f) \).
Using the chain and product rules for Fréchet derivatives and noting that \( u_{\lambda,0} = \lambda \), we have
\begin{equation*}
\begin{aligned}
	D\mathcal{N}_{\lambda,C}(0)f_{B} =\, & \partial_{t}([C(u_{\lambda,tB},\epsilon(u_{\lambda,tB})):\epsilon(u_{\lambda,tB})]\nu\vert_{t=0} \\
	= &[\partial_{t}(C(u_{\lambda,tB},\epsilon( u_{\lambda,tB}))\vert_{t=0}:\underbrace{\epsilon( u_{\lambda,0})}_{=0}]\nu \\
	&+ [\underbrace{C(u_{\lambda,0},\epsilon( u_{\lambda,0})}_{C(\lambda,0)}:\partial_{t}(\epsilon( u_{\lambda,tB}))\vert_{t=0}]\nu.
\end{aligned}
\end{equation*}
The proof is completed by recalling that $B^T=B$, so that $\nabla\left(\partial_t  u_{\lambda,tB}|_{t=0}\right)=\nabla (v_{\lambda,B})\equiv B$ and 
$$\begin{aligned}\partial_t\epsilon(  u_{\lambda,tB})|_{t=0}=\epsilon\left(\partial_t  u_{\lambda,tB}|_{t=0}\right)&=\frac{1}{2}\left(\nabla\left(\partial_t  u_{\lambda,tB}|_{t=0}\right)+\nabla\left(\partial_t  u_{\lambda,tB}|_{t=0}\right)^T\right)\\
&\equiv\frac{1}{2}(B+B^T)=B.\end{aligned}$$
\end{proof}

With Lemma \ref{lemma-solution-map-conormal} at hand, one can recover the full linearized tensor \( C(\lambda,0) \) from point measurements of the conormal derivative at finitely many boundary points of solutions with Dirichlet data in a finite dimensional space. For this purpose, we need the following intermediate result.

\begin{lemma}\label{lemma-matrix-inequality}
Let \( e_{1},\ldots,e_{n}\in\mathbb{R}^{n} \) be the canonical Euclidean basis and let \( f_{1},\ldots,f_{n}\in\mathbb{R}^{n} \) be any vectors satisfying
\begin{equation*}
	\norm{f_{i}-e_{i}}_{\mathbb{R}^{n}} < \frac{1}{\sqrt{n}},\quad\forall i\in\{1,\ldots,n\}.
\end{equation*}
Then it holds for any matrix \( A\in\mathbb{R}^{n\times n} \) that
\begin{equation*}
	(Af_{i} = 0,\  i\in\{1,\ldots,n\})\implies A = 0.
\end{equation*}
\end{lemma}
\begin{proof}
Define \( M \coloneqq {\displaystyle\max_{i\in\{1,\ldots,n\}}{\{\norm{f_{i}-e_{i}}_{\mathbb{R}^{n}}\}}} \) and note that \( M<\frac{1}{\sqrt n} \).
Then
\begin{equation*}
	\norm{A}_{\mathbb{R}^{n\times n}}^2 \leq\sum_{i=1}^n\norm{Ae_i}_{\mathbb{R}^{ n}}^2 \leq \sum_{i=1}^n\norm{A(e_i-f_{i})}_{\mathbb{R}^{ n}}^{2} \leq nM^{2}\norm{A}_{\mathbb{R}^{ n\times n}}^2
\end{equation*}
and since \( nM^{2}<1 \) it follows that \( \norm{A}_{\mathbb{R}^{n\times n}} = 0 \).
\end{proof}

Armed with Lemma \ref{lemma-matrix-inequality}, we are now in position to complete the proof of Theorem \ref{theorem-first-linearization} and Corollary \ref{c1}.

\begin{proof}[{\normalfont\bfseries Proof of Theorem \ref{theorem-first-linearization}.}]
Assume that \eqref{eq-first-linearization-1} is fulfilled  for some $\lambda\in\R^n$.  Let \( B\in\mathcal{W} \) and  $i\in\{1,\ldots,n\}$.
It follows from \eqref{eq-first-linearization-1} and Lemma \ref{lemma-solution-map-conormal} that
\begin{equation}\label{eq-first-linearization-2}
\begin{aligned}
	(C^1(\lambda,0):B)\nu(x_{i})=[D\mathcal{N}_{\lambda,C^{1}}(0)f_{B}](x_{i}) &= [D\mathcal{N}_{\lambda,C^{2}}(0)f_{B}](x_{i}) \\
    &=[C^2(\lambda,0):B]\nu(x_{i}).
\end{aligned}
\end{equation}
We wish to remove \( \nu(x_{i}) \) from this identity.
To this end, we apply \eqref{eq-normal-vector-condition} and Lemma \ref{lemma-matrix-inequality} to the matrix \( A \) and vectors \( f_{i} \) defined by
\begin{equation*}
	A \coloneqq [C^1(\lambda,0) - C^2(\lambda,0)]:B,\quad f_{i}\coloneqq \nu(x_{i})
\end{equation*}
and conclude that
\begin{equation*}
	 [C^1(\lambda,0)-C^2(\lambda,0)]:B = 0
\end{equation*}
for every \( B\in\mathcal{W} \).
In other words, \( [C^1(\lambda,0)-C^2(\lambda,0)]:B \) is the zero matrix, and we find in row \( i \) and column \( j \) of this matrix, after taking \( B = B_{k\ell} \) and using assumption \ref{assumption-symmetry}, that
\begin{equation*}
	C^1_{ijk\ell}(\lambda,0)=(C^1(\lambda,0):B_{k\ell})_{ij}=(C^2(\lambda,0):B_{k\ell})_{ij}=C^2_{ijk\ell}(\lambda,0),
\end{equation*}
where for any matrix \( A\in\R^{n\times n} \), \( (A)_{ij} \) denotes the coefficient at row \( i \) and column \( j \) of \( A \).
The proof is completed by repeating this for all \( k,\ell\in\{1,\ldots,n\} \).
\end{proof}

\begin{proof}[{\normalfont\bfseries Proof of Corollary \ref{c1}.}]
Assume that \eqref{c1a} is fulfilled. By following the proof of Theorem \ref{theorem-first-linearization}, we obtain
$$(C^1(\lambda,0):B_{kk})\nu(x_0)=(C^2(\lambda,0):B_{kk})\nu(x_0).
$$
On the other hand, we have
$$C^m(\lambda,0):B_{kk}=\Lambda^m(\lambda,0)Id_{\mathbb R^{n\times n}}+2\mu^m(\lambda,0)B_{kk},\quad m=1,2,$$
which implies that
\bel{c1d}(\Lambda^1(\lambda,0)-\Lambda^2(\lambda,0))\nu(x_0)+2(\mu^1(\lambda,0)-\mu^2(\lambda,0))\nu_k(x_0)e_k=0,\ee
where $(e_1,\ldots,e_n)$ denotes the canonical orthonormal basis of $\mathbb{R}^n$.
Recall that \( \nu(x_{0}) \) has two nonzero components, which we now denote \( \nu_{k}(x_{0}) \) and \( \nu_{\ell}(x_{0}) \).
With this in mind we obtain $\Lambda^1(\lambda,0)=\Lambda^2(\lambda,0)$ from component \( l \) of \eqref{c1d}, after which $\mu^1(\lambda,0)=\mu^2(\lambda,0)$ follows from component \( k \).
In conclusion, $C^1(\lambda,0)=C^2(\lambda,0)$.
\end{proof}

Now let us move to the stability estimates of  Theorem \ref{theorem-first-linearization-stability} and  Corollary \ref{c2}.

\begin{proof}[{\normalfont\bfseries Proof of Theorem \ref{theorem-first-linearization-stability}.}]
Let $R>0$, $\lambda\in[-R,R]^n$ and $B\in\mathcal W$.
For notational simplicity, denote by \( C(\lambda) \) the difference of the tensors
\begin{equation*}
	C(\lambda) = C^1(\lambda,0)-C^2(\lambda,0),\quad M \coloneqq \max_{i=1,\ldots,n}\{\norm{\nu(x_{i})-e_{i}}_{\mathbb{R}^{n}}\}
\end{equation*}
and, applying \eqref{eq-normal-vector-condition}, we recall that $M<\frac{1}{\sqrt n}$.
We start by proving the following estimate:
\begin{equation}\label{eq-first-linearization-stability-1}
	\norm{C(\lambda):B}_{\mathbb{R}^{n\times n}}\leq \frac{\sqrt{2n+2n^2M^2}}{1 -nM^2}\max_{i=1,\ldots,n}\norm{(C(\lambda):B)\nu(x_i)}_{\mathbb{R}^{n}}.
\end{equation}
Notice that
$$\begin{aligned}
	\norm{C(\lambda):B}^2_{\mathbb{R}^{n\times n}} &\leq\sum_{i=1}^n\norm{(C(\lambda):B)e_i}^2_{\mathbb{R}^{n}}\\
    &\leq \sum_{i=1}^n\left(\norm{(C(\lambda):B)(e_i-\nu(x_i))}_{\mathbb{R}^{n}}+\norm{(C(\lambda):B)\nu(x_i)}_{\mathbb{R}^{n}}\right)^2\\
    &\leq \sum_{i=1}^n\left(\norm{C(\lambda):B}_{\mathbb{R}^{n\times n}}M+\norm{(C(\lambda):B)\nu(x_i)}_{\mathbb{R}^{n}}\right)^2 \\
    &\leq n\left(\norm{C(\lambda):B}_{\mathbb{R}^{n\times n}}M+\max_{i=1,\ldots,n}\norm{(C(\lambda):B)\nu(x_i)}_{\mathbb{R}^{n}}\right)^2.\end{aligned}$$
Define
\begin{equation}\label{eq-first-linearization-stability-2}
	p = \max_{i=1,\ldots,n}\norm{(C(\lambda):B)\nu(x_i)}_{\mathbb{R}^{n}},\qquad q = \norm{C(\lambda):B}_{\mathbb{R}^{n\times n}}.
\end{equation}
Then the above inequality can be written \( q^{2} \leq n(Mq + p)^{2} \).
By using the inequality \( 2ab\leq a^{2}+b^{2} \) with \( a=\varepsilon q \) and \( b=\frac{nMp}{\varepsilon} \) we have
\begin{equation*}
	q^{2} \leq (n+\frac{n^{2}M^{2}}{\varepsilon^{2}})p^{2} + (\varepsilon^{2} + nM^{2})q^{2}.
\end{equation*}
Set \( \varepsilon = \sqrt{\frac{1-nM^{2}}{2}} \), so that \( \varepsilon^{2}+nM^{2}<1 \), and absorb the \( q^{2} \) term on the left-hand side.
It follows that 
\begin{equation*}
	q^{2}\leq\frac{n+\frac{n^{2}M^{2}}{\varepsilon^{2}}}{1-\varepsilon^{2}-nM^{2}}p^{2} = 2n\frac{1+nM^{2}}{(1-nM^{2})^{2}}p^{2}.
\end{equation*}
From this estimate we deduce \eqref{eq-first-linearization-stability-1}. Then, we obtain the stability estimate
\begin{equation*}
\begin{aligned}
	&\sup_{\lambda\in[-R,R]^n}\norm{C(\lambda)}_{*} = \sup_{\lambda\in[-R,R]^n}\sup_{\substack{A\in S_{n}(\mathbb{R})\\\norm{A}_{\mathbb{R}^{n\times n}}=1}}\norm{C(\lambda):A}_{\mathbb{R}^{n\times n}} \\
	&= \sup_{\lambda\in[-R,R]^n}\sup_{\substack{B_{ij}\in\mathcal{W}\\\theta_{ij}\in\mathbb{R}\\\sum_{ij}\abs{\theta_{ij}}^{2}=1}}\norm{\sum_{ij}C(\lambda):\theta_{ij}B_{ij}}_{\mathbb{R}^{n\times n}} \\
	&\leq \frac{n(n+1)}{2} \max_{B\in\mathcal{W}}\sup_{\lambda\in[-R,R]^n}\norm{C(\lambda):B}_{\mathbb{R}^{n\times n}} \\
	&\leq \frac{\frac{n(n+1)}{2}\sqrt{2n(1+nM^{2})}}{1-nM^{2}}\max_{B\in\mathcal{W}}\sup_{\lambda\in[-R,R]^n}\max_{i=1,\ldots,n}\norm{(C(\lambda):B)\nu(x_{i})}_{\mathbb{R}^{n}}.
\end{aligned}
\end{equation*}
The theorem follows by recalling that, for all $B\in \mathcal W$, we have
\begin{equation*}
	(C(\lambda):B)\nu(x_{i}) = [D\mathcal{N}_{\lambda,C^{1}}(0)f_{B}](x_{i})- [D\mathcal{N}_{\lambda,C^{2}}(0)f_{B}](x_{i}),\quad i=1,\ldots,n,
\end{equation*}
by Lemma \ref{lemma-solution-map-conormal}.
\end{proof}

\begin{proof}[{\normalfont\bfseries Proof of Corollary \ref{c2}.}]
Following the proof of Corollary \ref{c1}, for $m=1,2$, we obtain
\begin{equation*}
\begin{aligned}
	[D\mathcal{N}_{\lambda,C^{m}}(0)f_{B_{kk}}](x_{0}) &=(C^m(\lambda,0):B_{kk})\nu(x_0)\\
	&=\Lambda^m(\lambda,0)\nu(x_0)+2\mu^m(\lambda,0)\nu_k(x_0)e_k.
\end{aligned}
\end{equation*}
Recall that \( \nu(x_{0}) \) has at least two nonzero components and denote two of them by \( \nu_{k}(x_{0}),\nu_{\ell}(x_{0}) \).
For notational simplicity, let us write \( \Lambda^{i},\mu^{i},D\mathcal{N}^{i} \) and \( \nu \) in place of \( \Lambda^{i}(\lambda,0),\mu^{i}(\lambda,0), [D\mathcal{N}_{\lambda,C^{i}}(0)f_{B_{kk}}](x_{0}),\nu(x_{0}) \).
Then we have
\begin{equation}\label{eq-theorem-stability-1}
	D\mathcal{N}^{1}-D\mathcal{N}^{2} =\nu(\Lambda^{1}-\Lambda^{2}) + 2\nu_{k}(\mu^{1}-\mu^{2}).
\end{equation}
From the \( l \)-component of \eqref{eq-theorem-stability-1} we obtain
\begin{equation}\label{eq-theorem-stability-2}
	\abs{\Lambda^{1}-\Lambda^{2}}\leq \frac{1}{\abs{\nu_{\ell}}}\norm{D\mathcal{N}^{1}-D\mathcal{N}^{2}}_{\mathbb{R}^{n}},
\end{equation}
and from component \( k \) of \eqref{eq-theorem-stability-1} we find
\begin{equation}\label{eq-theorem-stability-3}
\begin{aligned}
	\abs{\nu_{k}}\cdot\abs{\mu^{1}-\mu^{2}} &\leq \abs{D\mathcal{N}^{1}-D\mathcal{N}^{2}} + \abs{\nu_{k}}\cdot\abs{(\Lambda^{1}-\Lambda^{2})+(\mu^{1}-\mu^{2})} \\
	&\leq \abs{D\mathcal{N}^{1}-D\mathcal{N}^{2}} + 2\abs{\nu_{k}}\cdot\abs{(\Lambda^{1}-\Lambda^{2})+(\mu^{1}-\mu^{2})}.
\end{aligned}
\end{equation}
By using the triangle inequality, \eqref{eq-theorem-stability-1} and \eqref{eq-theorem-stability-2}, we find
\begin{equation*}
\begin{aligned}
	&2\abs{\nu_{k}}\cdot\lvert(\Lambda^{1}-\Lambda^{2})+(\mu^{1}-\mu^{2})\rvert \\
	&\leq \abs{\nu_{k}(\Lambda^{1}-\Lambda^{2})+2\nu_{k}(\mu^{1}-\mu^{2})} + \abs{\nu_{k}(x_{0})}\cdot\abs{\Lambda^{1}-\Lambda^{2}} \\
	&= \abs{D\mathcal{N}^{1}-D\mathcal{N}^{2}} + \abs{\nu_{k}}\cdot\abs{\Lambda^{1}-\Lambda^{2}} \\
	&\leq \bigg(1+\frac{\abs{\nu_{k}}}{\abs{\nu_{\ell}}}\bigg)\abs{D\mathcal{N}^{1}-D\mathcal{N}^{2}}.
\end{aligned}
\end{equation*}
The result follows by inserting this in \eqref{eq-theorem-stability-3} and after dividing by \( \abs{\nu_{k}} \).
\end{proof}

\section{Recovery of dependence on strain tensor}\label{sec5}
We study the recovery of the \( \eta \)-dependence of the elastic tensor
\begin{equation*}
	\mathbb{R}^{n}\times S_n(\mathbb R)\ni(\lambda,\eta)\mapsto C(\lambda,\eta)\in \mathbb{R}^{n\times n\times n\times n}.
\end{equation*}
Our strategy is based on explicit asymptotic properties of the data $\mathcal N_{\lambda,C}(tf_B)$ as $t\to0$, for $B\in\mathcal M$ and $\lambda\in\R^n$. In the preceding section, we used this asymptotic  property at order $1$ for proving the recovery of $C(\lambda,0)$. Under the extra regularity assumption that \( C\in C^{k+1}(\R^n\times S_n(\R), \R^{n\times n\times n\times n})\), $k\in\mathbb N$, we will extend this asymptotic property for determining the coefficients  \( D_{\eta}^{j}C(\lambda,0) \), $j=0,\ldots,k-1$, in the Taylor polynomial of $C$ with respect to $\eta$ as $\eta=0$. 

According to Proposition \ref{lemma-solution-map}, for \( C\in C^{k+1}(\R^n\times S_n(\R), \R^{n\times n\times n\times n}\) and for every $\lambda\in\R^n$, the map \( \mathcal{N}_{\lambda,C} \) is \( C^{k} \) on a neighborhood $U_\lambda\subset W^{2-1/p,p}(\partial\Omega,\R^n)$ of zero as a map taking values in $ C(\partial\Omega,\R^n)$. Moreover, the coefficients in the Taylor polynomial at $t=0$ of the map $t\mapsto \mathcal{N}_{\lambda,C}(tf_B)$ are given by \( \partial_t^j\mathcal{N}_{\lambda,C}(tf_B)|_{t=0}=D^{j}\mathcal{N}_{\lambda,C}(0)(f_{B},\ldots,f_{B}) \), $j=0,\ldots,k$. In the next theorem we establish an explicit relation between the derivatives \( D^{k}\mathcal{N}_{\lambda,C}(0)(f_{B},\ldots,f_{B}) \) and  the derivatives \( D_{\eta}^{k-1}C(\lambda,0) \), which will play a fundamental role in the determination of the \( \eta \)-dependence of \( C(\lambda,\eta) \).

\begin{theorem}\label{l2}
Let \( \Omega \) satisfy assumption \ref{assumption-domain} and, for $j=1,2$, let \( C=C^j \) be two tensors satisfying assumptions \ref{assumption-nonlinearity}, with $k\geq2$,  \ref{assumption-ellipticity} and \ref{assumption-symmetry}.
Let $\lambda\in\R^n$ and \( B\in S_{n}(\mathbb{R}) \). Then, we have
\begin{equation}\label{eq-higher-linearization-9}
\begin{aligned}
	&D^{k}\mathcal{N}_{\lambda,C^1}(0)(f_{B},\ldots, f_{B}) -D^{k}\mathcal{N}_{\lambda,C^2}(0)(f_{B},\ldots, f_{B}) \\
	&= F_{k} + [(k-1)[D_{\eta}^{k-1}[C^1(\lambda,0)-C^2(\lambda,0)](B,\ldots,B)]:B]\nu,
\end{aligned}
\end{equation}
where \( F_{k} \) satisfies
\begin{equation}\label{eq-higher-linearization-4}
(D_{\eta}^{\ell}C^1(\lambda,0) = D_{\eta}^{\ell}C^2(\lambda,0),\ \lambda\in(-R,R)^n,\ \ell\leq k-2)
\implies F_{k} = 0.
\end{equation}
\end{theorem}

\begin{proof}
For $j=1,2$, let \( S_j \) be the the solution map from Lemma \ref{lemma-solution-map}, with $C=C^j$, such that \( S_j(0) = \lambda \). Following Lemma \ref{lemma-solution-map},  the map $t\mapsto \mathcal{N}_{\lambda,C^{j}}(tf_B)$, $j=1,2$,  is a $C^k$ on a sufficiently small neighborhood of zero, as a map taking values in $C(\partial\Omega,\R^n)$
and, for all $m=2,\ldots,k$, we can consider  the following boundary value problem
\bel{eq6}
\left\{
\begin{array}{ll}
\operatorname{div}(C^j(\lambda,0):\epsilon( Y^{(m)}_{j,\lambda})) =\operatorname{div} \mathcal K_{j,\lambda,B}^m  & \mbox{in}\ \Omega ,
\\
Y^{(m)}_{j,\lambda,B}=0 &\mbox{on}\ \partial\Omega,
\end{array}
\right.
\ee
with
\bel{kk}
\begin{aligned}
	\mathcal K_{j,\lambda,B}^m=&\partial_t^m (C^j(S_j(tf_B),\epsilon( S_j(tf_B))):\varepsilon(S_j(tf_B))) |_{t=0} \\
	&-(C^j(\lambda,0):\partial_t^m\varepsilon( S_j(tf_B))|_{t=0}).
\end{aligned}
\ee
We consider here differentiations in $t$ in the sense of Banach space-valued functions taking values in $W^{1,p}(\Omega;\mathbb R^{n\times n})$.
In view of Corollary \ref{corollary-linear-well-posedness-symmetric}, we have $Y^{(m)}_{j,\lambda}\in W^{2,p}(\Omega;\mathbb R^{n})$ and,
repeating the arguments of Lemma \ref{lemma-solution-map-conormal}, one can observe that $\partial_t^mS_j(tf_B)|_{t=0}=Y^{(m)}_{j,\lambda,B}$ by the uniqueness of solutions of \eqref{eq6}.

We will prove formula \eqref{eq-higher-linearization-9}-\eqref{eq-higher-linearization-4} by iteration.\\
\textbf{Step 1:} Consider $k=2$ and let $\lambda\in(-R,R)^n$. 
For $j=1,2$, recall that  $\nabla S_j(0)=\nabla\lambda\equiv 0$, $\nabla\left(\partial_t  S_j(tf_B)|_{t=0}\right)=\partial_t\epsilon(   S_j(tf_B))|_{t=0}\equiv B$.
It follows
$$\begin{aligned} &\partial_t^2(\mathcal{N}_{\lambda,C^j}(tf_B))|_{t=0}\\
&=2(\partial_tC^j( S_j(tf_B),\epsilon( S_j(tf_B)))|_{t=0}:\varepsilon (\partial_t S_j(tf_B))|_{t=0})\nu ) +(C^j(\lambda,0):\varepsilon (\partial_t^2S_j(tf_B))|_{t=0})\nu\\
&=2(D_\eta C^j(\lambda,0)(B):B)\nu+2(D_\lambda C^j(\lambda,0)(B):B)\nu +(C^j(\lambda,0):\varepsilon(Y^{(2)}_{j,\lambda}))\nu.\end{aligned}$$
From this identity, we obtain \eqref{eq-higher-linearization-9} with the term $F_2$  given by 
$$F_2=\sum_{j=1}^2(-1)^{j+1}\left[2(D_\lambda C^j(\lambda,0)(B):B)\nu +(C^j(\lambda,0):\varepsilon(Y^{(2)}_{j,\lambda}))\nu\right].$$
It remains to prove the implication \eqref{eq-higher-linearization-9} for $k=2$. For this purpose, assume that $	C^1(\lambda,0) = C^2(\lambda,0)$,  $\lambda\in(-R,R)^n$. Then, fixing $\lambda\in(-R,R)^n$,  we have
\bel{ff}F_2=(C^1(\lambda,0):\varepsilon(Y^{(2)}_{1,\lambda}-Y^{(2)}_{2,\lambda}))\nu.\ee
Recall that
 $Y^{(2)}_{j,\lambda}$ solves \eqref{eq6} with $\mathcal K^2_{j,\lambda}$ given by \eqref{kk}
with $m=2$. From the above calculation, we obtain
$$\mathcal K^2_{j,\lambda}=2(D_\eta C^j(\lambda,0)(B):B)+2(D_\lambda C^j(\lambda,0)(f_B):B)$$
which implies that 
$$\operatorname{div}(\mathcal K^2_{1,\lambda})=2\operatorname{div}((D_\lambda C^1(\lambda,0)(f_B):B))=2\operatorname{div}((D_\lambda C^2(\lambda,0)(f_B):B))=\operatorname{div}(\mathcal K^2_{2,\lambda}).$$ 
Combining this with the well-posedness of \eqref{eq6} (see e.g. Corollary \ref{corollary-linear-well-posedness-symmetric}) we obtain $Y^{(2)}_{1,\lambda}= Y^{(2)}_{2,\lambda}$.  Thus, we have $F_2=0$ and \eqref{eq-higher-linearization-9} holds true for $k=2$.

~\newline
\textbf{Step 2:} Let $N\geq 2$. Assume the assertion \eqref{eq-higher-linearization-9}-\eqref{eq-higher-linearization-4} as well as the condition
\begin{equation}\label{y}
\begin{aligned}&(D_{\eta}^{\ell}C^1(\lambda,0) = D_{\eta}^{\ell}C^2(\lambda,0),\ \lambda\in(-R,R)^n,\ \ell\leq k-2)\\
&\implies (Y^{(m)}_{1,\lambda}= Y^{(m)}_{2,\lambda},\  \lambda\in(-R,R)^n,\  m=2,\ldots, k)\end{aligned}
\end{equation}
hold true for $k=N$. Let us prove that the same properties are true at rank $k=N+1$. 

For $j=1,2$, an application of the chain rule yields
$$\begin{aligned} &\partial_t^{N+1}(\mathcal{N}_{\lambda,C^j}(tf_B))|_{t=0}\\
&=\partial_t^{N+1}((C^j(S_j(tf_B),\epsilon( S_j(tf_B))):\epsilon( S_j(tf_B)))\nu )|_{t=0}\\
&=\sum_{k=0}^{N+1}\binom{N+1}{k}(\partial_t^{N+1-k}C^j(S_j(tf_B),\epsilon( S_j(tf_B)))|_{t=0}:\epsilon (\partial_t^{k}S_j(tf_B))|_{t=0})\nu ).\end{aligned}$$
Combining this with the fact that  $\nabla S_j(tf_B)|_{t=0}=\nabla\lambda\equiv 0$,  $\partial_t^mS_j(tf_B)|_{t=0}=Y^m_{j,\lambda}$, $m=2,\ldots,N$, and $\partial_t\epsilon( S_j(tf_B))|_{t=0}=B $, we obtain
$$\begin{aligned} &\partial_t^{N+1}(\mathcal{N}_{\lambda,C^j}(tf_B))|_{t=0}\\
&=\sum_{k=1}^{N+1}\binom{N+1}{k}(\partial_t^{N+1-k}C^j(S_j(tf_B),\epsilon( S_j(tf_B)))|_{t=0}:\epsilon (\partial_t^{k}S_j(tf_B))|_{t=0})(x)\nu(x) )\\
&=(N+1)(\partial_t^{N}C^j(S_j(tf_B),\epsilon( S_j(tf_B)))|_{t=0}:B)\nu )+G_j\nu +(C^j(\lambda,0):\varepsilon(Y^{(N+1)}_{j,\lambda}))\nu,\end{aligned}$$
where \( G_{j} \)  satisfies the condition
$$
\begin{cases}
	D_{\eta}^{\ell}C^1(\lambda,0) = D_{\eta}^{\ell}C^2(\lambda,0)& \lambda\in(-R,R)^n,\ \ell\leq N-1,\\
	Y^{(m)}_{1,\lambda}=Y^{(m)}_{2,\lambda}& \lambda\in(-R,R)^n,\ m=2,\ldots,N,
\end{cases}
\implies G_1=G_2.
$$
On the other hand, applying \eqref{y}, we get
\begin{equation*}
\begin{aligned}
&(D_{\eta}^{\ell}C^1(\lambda,0) = D_{\eta}^{\ell}C^2(\lambda,0)\ \lambda\in(-R,R)^n,\ \ell\leq N-1) \\
&\implies (Y^{(\ell)}_{1,\lambda}=Y^{(\ell)}_{2,\lambda}\ \lambda\in(-R,R)^n,\ \ell\leq N) 
\end{aligned}
\end{equation*}
and it follows
\begin{equation}\label{GG}
(D_{\eta}^{\ell}C^1(\lambda,0) = D_{\eta}^{\ell}C^2(\lambda,0),\ \lambda\in(-R,R)^n,\ \ell\leq N-1)
\implies G_1=G_2.
\end{equation}
Meanwhile, repeating the above argumentation, we find
$$\begin{aligned}&\partial_t^{N}C^j(S_j(tf_B),\epsilon( S_j(tf_B)))|_{t=0}=D_\eta^{N}C^j(\lambda,0)(B^N)+\mathcal R_j,\end{aligned}$$
where \( \mathcal R_{j} \) is a sum of mixed derivatives \( D_{\lambda}^{i}D_{\eta}^{\ell}C^j(\lambda,0) \) with \( i+\ell = N \) and \( \ell\leq N-1 \).
Therefore, we have
\begin{equation}\label{RR}
(D_{\eta}^{\ell}C^1(\lambda,0) = D_{\eta}^{\ell}C^2(\lambda,0),\ \forall\abs{\lambda}<R,\ \ell\leq N-1)
\implies \mathcal R_1=\mathcal R_2.
\end{equation}
Then, the identity \eqref{eq-higher-linearization-9}, with $k=N+1$, holds true with
$$F_{N+1}=(G_1-G_2)\nu +(\mathcal R_1-\mathcal R_2)\nu+(C^1(\lambda,0):\varepsilon(Y^{(N+1)}_{1,\lambda}))\nu-(C^2(\lambda,0):\varepsilon(Y^{(N+1)}_{2,\lambda}))\nu.$$
We only need to check the implication \eqref{eq-higher-linearization-4} for $k=N+1$. For this purpose, let us assume that $D_{\eta}^{\ell}C^{1}(\lambda,0) = D_{\eta}^{\ell}C^{2}(\lambda,0),\ \lambda\in(-R,R)^n,\ \ell\leq N-1$ and observe from \eqref{GG}-\eqref{RR} that
$$F_{N+1}=(C^1(\lambda,0):\varepsilon(Y^{(N+1)}_{1,\lambda}-Y^{(N+1)}_{2,\lambda}))\nu.$$
Thus, the proof will be completed if we show that $Y^{(N+1)}_{1,\lambda}=Y^{(N+1)}_{2,\lambda}$. For this purpose, recall that $Y^{N+1}_{j,\lambda}$, $j=1,2$, solves \eqref{eq6}, with $m=N+1$, $C=C^j$ and $\mathcal K_{j,\lambda}^{N+1}$, given by \eqref{kk} with $m=N+1$. On the other hand, using the above properties, we find
$$\begin{aligned} &\mathcal K_{j,\lambda}^{N+1}\\
&=\sum_{k=1}^{N}\binom{N+1}{k}(\partial_t^{N+1-k}C^j(S_j(tf_B),\epsilon( S_j(tf_B)))|_{t=0}:\epsilon (S_j(tf_B))|_{t=0})\\
&=(N+1)D_\eta^{N}C^j(\lambda,0)(B^N):B+H_{\lambda}^{N+1},\end{aligned}$$
where $H^{N+1}_{\lambda}$ is independent of $j$. Therefore, we obtain
$$\operatorname{div}(\mathcal K_{1,\lambda}^{N+1})=\operatorname{div}(H_{\lambda}^{N+1})=\operatorname{div}(\mathcal K_{2,\lambda}^{N+1})$$
and $Y^{(N+1)}_{1,\lambda}=Y^{(N+1)}_{2,\lambda}$ now follows from the well-posedness of \eqref{eq6}. This proves that the property  is true at rank $N+1$ and it completes the proof of the theorem.
\end{proof}

We now proceed to the proofs of the main results on determination of the dependence on the strain tensor.
The key ingredient  is  to determine \( D_{\eta}^{k-1}C(\lambda,0) \) from the coefficient \( D^{k}\mathcal{N}_{\lambda,C}(0)(f_B,\ldots,f_B) \) of order $k$ in the Taylor polynomial of the map $t\mapsto\mathcal{N}_{\lambda,C}(tf_B) $ at $t=0$. We use different strategies in Theorems \ref{t4} and \ref{t5}, where the measurements are made in \( n \) boundary points \( x_{1},\ldots,x_{n} \), and Theorem \ref{c3}, where the measurements are restricted to the single point \( x_{0} \).

\begin{proof}[{\normalfont\bfseries Proof of Theorem \ref{t4}}]
Assume that \eqref{t4ab} holds true. By using \eqref{t4a} and \( \mu_{ij}^{1}(\lambda,0) = \mu_{ij}^{2}(\lambda,0) = 1 \) together with Theorem \ref{theorem-first-linearization} we conclude that
	$D_{ijk\ell}^{1}(\lambda) = D_{ijk\ell}^{2}(\lambda)=D_{ijk\ell}(\lambda)$, $\lambda\in[-R,R]^n$.

We will show \eqref{t4c} by iteration. Fix $p\in\{0,\ldots,N-2\}$ and assume that
 $$D_\eta^kC^1(\lambda,0) = D_\eta^kC^2(\lambda,0),\quad k=0,1,\ldots,p,\ \lambda\in[-R,R]^n.$$
 By taking into account \eqref{t4a}, for all $i,j\in\{1,\ldots,n\}$, we have 
 \begin{equation}\label{t4f}
 	D_\eta^k\mu_{ij}^1(\lambda,0)D_{ijij}(\lambda) = D_\eta^k\mu_{ij}^2(\lambda,0)D_{ijij}(\lambda),\quad k=0,\ldots,p,\ \lambda\in[-R,R]^n
 \end{equation}
 and we need to show that
 \begin{equation}\label{t4g}
 	D_\eta^{p+1}\mu_{ij}^1(\lambda,0) = D_\eta^{p+1}\mu_{ij}^2(\lambda,0),\quad  \lambda\in[-R,R]^n
 \end{equation}
 holds true.
 By continuity of the map $\mathbb R^n\ni\lambda\mapsto D^{p+1}C(\lambda,0)$, it is enough to show \eqref{t4g} for \( \lambda\in(-R,R)^{n} \).
 For this purpose, fix $\lambda\in(-R,R)^n$. 
In view of condition \eqref{t4ab}, for all $B\in \mathcal W_\star^{N-1}$ and all $i\in\{1,\ldots,n\}$, we obtain
$$\begin{aligned}D^{p+2}\mathcal N_{\lambda,C^1}(f_B,\ldots,f_B)(x_i)=\partial_t^{p+1}\mathcal N_{\lambda,C^1}(tf_B)(x_i)|_{t=0}&=\partial_t^{p+1}\mathcal N_{\lambda,C^2}(tf_B)(x_i)|_{t=0}\\
&=D^{p+2}\mathcal N_{\lambda,C^2}(f_B,\ldots,f_B)(x_i).\end{aligned}$$
 Combining this with Theorem \ref{l2},  for all $B\in \mathcal W_\star^{N-1}$ and all $i\in\{1,\ldots,n\}$, we find
 \bel{t4h}[D_\eta^{p+1}C^1(\lambda,0)(B,\ldots,B):B]\nu(x_i)=[D_\eta^{p+1}C^2(\lambda,0)(B,\ldots,B):B]\nu(x_i).\ee
Let \( f_{i}\coloneqq \nu(x_{i}) \) and 
\begin{equation*}
 	A \coloneqq [D_\eta^{p+1}C^1(\lambda,0)(B,\ldots,B) - D_\eta^{p+1}C^2(\lambda,0)(B,\ldots,B)]:B.
\end{equation*}
Then \eqref{t4h} implies \( Af_{i} = 0 \), $i=1,\ldots,n$.  Therefore, applying \eqref{eq-normal-vector-condition} and Lemma \ref{lemma-matrix-inequality},  obtain
 \begin{equation}\label{t4i}
	 D_\eta^{p+1}C^1(\lambda,0)(B,\ldots,B):B = D_\eta^{p+1}C^2(\lambda,0)(B,\ldots,B):B,\quad B\in \mathcal W_\star^{N-1}.
 \end{equation}
Moreover, fixing $(i,j)\in\{1,\ldots,n\}^2$ and applying \eqref{t4a}, for all $B\in \mathcal W_\star^{N-1}$, we find
 \begin{equation*}
 \begin{aligned}
 [D_\eta^{p+1}C^m(\lambda,0)(B,\ldots,B):B]_{ij}&=\sum_{k,\ell=1}^nD_\eta^{p+1}C^m_{ijk\ell}(\lambda,0)(B,\ldots,B)b_{k\ell}\\
 &=D_\eta^{p+1}\mu^m_{ij}(\lambda,0)(B,\ldots,B)D_{ijij}(\lambda)b_{ij},
 \end{aligned}
 \end{equation*}
 where, for any matrix $A=(a_{k\ell})_{1\leq k,\ell\leq n}\in \R^{n\times n}$, $(A)_{ij}$ denotes the coefficient $a_{ij}$.
 Combining this with \eqref{t4i}, we obtain  
 \bel{t4ii}
 [D_\eta^{p+1}\mu^1_{ij}(\lambda,0)-D_\eta^{p+1}\mu^2_{ij}(\lambda,0)](B,\ldots,B)D_{ijij}(\lambda)b_{ij}=0 ,\  B\in \mathcal W_\star^{N-1},\ee
 where \( B =(b_{ij})_{1\leq i,j\leq n} \).
Now define
$$\mathcal V_{ij}:=\{B_{ij}\}\cup\{B_{ij}+B_{k\ell}:\ k,\ell=1,\ldots,n,\ (k,\ell)\neq (i,j)\}\subset \mathcal W_\star,$$
$$\mathcal V_{ij}^{N-1}:=\{B_1+\ldots+B_k:\ k=1,\ldots,N-1,\ B_1,\ldots,B_k\in\mathcal \mathcal V_{ij}\}\subset \mathcal W_\star^{N-1}$$
and notice that $\mathcal V_{ij}$ is a basis of $S_n(\mathbb R)$. Condition \eqref{t4ii} implies
$$
[D_\eta^{p+1}\mu^1_{ij}(\lambda,0)-D_\eta^{p+1}\mu^2_{ij}(\lambda,0)](B,\ldots,B)D_{ijij}(\lambda)=0 ,\quad B\in \mathcal V_{ij}^{N-1}.$$
Moreover, applying assumption \ref{assumption-ellipticity}, we find
$$ D_{ijij}(\lambda)=E_{ij}:C(\lambda,0):E_{ij}\geq \kappa(\lambda)>0$$
and it follows
\bel{t4cc}
[D_\eta^{p+1}\mu^1_{ij}(\lambda,0)-D_\eta^{p+1}\mu^2_{ij}(\lambda,0)](B,\ldots,B)=0 ,\quad B\in \mathcal V_{ij}^{N-1}.
\ee
Now let us observe that, using a polarization identity \cite[Theorem 1 \& equation (7)]{T}, the expressions 
$$[D_\eta^{p+1}\mu^1_{ij}(\lambda,0)-D_\eta^{p+1}\mu^2_{ij}(\lambda,0)](B_1,\ldots,B_{p+1}) ,\quad B_1,\ldots,B_{p+1}\in \mathcal V_{ij}$$
are determined by 
$$[D_\eta^{p+1}\mu^1_{ij}(\lambda,0)-D_\eta^{p+1}\mu^2_{ij}(\lambda,0)](B,\ldots,B),\quad B\in \mathcal V_{ij}^{N-1}.$$
Thus, condition \eqref{t4cc} implies
\bel{t4ccc}
	[D_\eta^{p+1}\mu^1_{ij}(\lambda,0)-D_\eta^{p+1}\mu^2_{ij}(\lambda,0)](B_1,\ldots,B_{p+1})=0 ,\quad B_1,\ldots,B_{p+1}\in \mathcal V_{ij}.
\ee
Since \( D_\eta^{p+1}\mu^1_{ij}(\lambda,0)-D_\eta^{p+1}\mu^2_{ij}(\lambda,0) \) is a tensor of rank \( p+1 \), it follows that \eqref{t4ccc} holds when \( B_{1},\ldots,B_{p+1}\in \mathcal{V}_{ij} \) are replaced by \( U_{1},\ldots,U_{p+1}\in S_{n}(\mathbb{R}) \).
To see this, use the fact that \( \mathcal{V}_{ij} \) is a basis and write \( U_{\ell} = \sum_{e}\alpha_{l,e}B_{e} \) for \( \alpha_{l,e}\in\mathbb{R} \), \( B_{e}\in\mathcal{V}_{ij} \), \( l=1,\ldots,p+1 \), \( e=1,\ldots,\frac{n(n+1)}{2} \).
Then
\begin{equation*}
\begin{aligned}
	&[D_\eta^{p+1}\mu^1_{ij}(\lambda,0)-D_\eta^{p+1}\mu^2_{ij}(\lambda,0)](U_1,\ldots,U_{p+1})\\
	&=\sum_{e_1,\ldots,e_{p+1}}\alpha_{1,e_{1}}\cdot\ldots\cdot\alpha_{p+1,e_{p+1}}[D_\eta^{p+1}\mu^1_{ij}(\lambda,0)-D_\eta^{p+1}\mu^2_{ij}(\lambda,0)](B_{e_1},\ldots,B_{e_{p+1}}).
\end{aligned}
\end{equation*}
In view of \eqref{t4ccc} each term in the sum vanishes and we conclude that
\begin{equation*}
	[D_\eta^{p+1}\mu^1_{ij}(\lambda,0)-D_\eta^{p+1}\mu^2_{ij}(\lambda,0)](U_1,\ldots,U_{p+1})=0,\quad U_{1},\ldots,U_{p+1}\in S_{n}(\mathbb{R}).
\end{equation*}
The proof is completed by repeating this argument for every pair of indices \( i,j \).
\end{proof}

\begin{proof}[{\normalfont\bfseries Proof of Theorem \ref{t5}}] Let \eqref{t5ab} be fulfilled.
Proceeding as in the proof of Theorem \ref{t4}, we first conclude that \( D_{ijk\ell}^{1}(\lambda) = D_{ijk\ell}^{2}(\lambda)=D_{ijk\ell}(\lambda) \), $i,j,k,\ell=1,\ldots,n$, $\lambda\in[-R,R]^n$.  We will prove the result by induction. For this purpose, we fix $p\in\{0,\ldots,N-2\}$ and we assume that $D_\eta^kC^1(\lambda,0) = D_\eta^kC^2(\lambda,0),\quad k=0,1,\ldots,p,\ \lambda\in[-R,R]^n$. Let $i,j\in\{1,\ldots,n\}$ be fixed and consider, for all $k,\ell=1,\ldots,n$, the matrix
\begin{equation*}
	H_{k\ell} =
\begin{cases}
	B_{k\ell}&\text{if }D_{ijk\ell}(\lambda) >0, \\
    -B_{k\ell}&\text{if }D_{ijk\ell}(\lambda) >0, \\
	B_{ij}+B_{k\ell}&\text{if }D_{ijk\ell}(\lambda) =0.
\end{cases}
\end{equation*}
Notice that all possible values of $H_{k\ell}$ belong to \( \mathcal{W}_{\diamond} \).
This choice of $H_{k\ell}$ ensures that
\begin{equation}\label{eq-t5a}
	\sum_{k,\ell}D_{ijk\ell}(\lambda)H_{k\ell}>0,\quad i,j=1,\ldots n.
\end{equation}
 Now let us consider
\( \mathcal{U}_{ij} = \{H_{k\ell}: 1\leq k,\ell \leq n\}\subset \mathcal W_{\diamond} \) and observe that $ \mathcal{U}_{ij}$ is a basis of $S_n(\R)$.
By the same polarization argument as in the proof of Theorem \ref{t4} it is enough to identify the derivatives \( D_{\eta}^{p+1}\mu_{ij}^{m}(\lambda,0)(B,\ldots,B) \) for \( B\in \mathcal{U}_{ij}^{p+1} \) where
\begin{equation*}
	\mathcal{U}_{ij}^{p+1} = \{A_{1}+\ldots + A_{k}: 1\leq k\leq p+1,\,\,A_{1},\ldots,A_k\in \mathcal{U}_{ij}\}
.\end{equation*}
We proceed with the same induction argument as in Theorem \ref{t4} and obtain for every \( B=(b_{k\ell})_{1\leq k,\ell\leq n}\in \mathcal{U}_{ij}^{p+1} \) that
\begin{equation*}
	\sum_{k,\ell=1}^n[D_{\eta}^{p+1}C_{ijk\ell}^{1}(\lambda,0)-D_{\eta}^{p+1}C_{ijk\ell}^{2}(\lambda,0)](B,\ldots,B)b_{k\ell} = 0.
\end{equation*}
Due to \eqref{t5a} this reduces to
\begin{equation*}
	[D_{\eta}^{p+1}\mu_{ij}^{1}(\lambda,0)-D_{\eta}^{p+1}\mu_{ij}^{2}(\lambda,0)](B,\ldots,B)\sum_{k,\ell=1}^nD_{ijk\ell}(\lambda)b_{k\ell} = 0,
\end{equation*}
and, applying \eqref{eq-t5a}, we conclude that
\begin{equation*}
	D_{\eta}^{p+1}\mu_{ij}^{1}(\lambda,0)(B,\ldots,B)=D_{\eta}^{p+1}\mu_{ij}^{2}(\lambda,0)(B,\ldots,B),\quad  B\in\mathcal{U}_{ij}^{p+1}.
\end{equation*}
The proof is completed by repeating this for every pair of indices \( i,j \).
\end{proof}

\begin{proof}[{\normalfont\bfseries Proof of Theorem \ref{c3}.}]
Assume that \eqref{c3bc} holds true. For simplicity of notation, we denote by \( \nu \) and \( \nu_{j} \) the normal vector \( \nu(x_{0}) \) and the \( j \)th component \( \nu_{j}(x_{0}) \).
The signs of the two nonzero components \( \nu_{k},\nu_{\ell} \) of the normal vector \( \nu \) play a crucial role in the proof.
In particular, if \( \nu_{k},\nu_{\ell} \) have different signs then we let \( k \) denote the larger component,
\begin{equation}\label{eq-k-larger}
	\abs{\nu_{k}}>\abs{\nu_{\ell}}
\end{equation}
and if \( \nu_{k},\nu_{\ell} \) have the same sign then we let \( l \) denote the larger component,
\begin{equation}\label{eq-l-larger}
	\abs{\nu_{\ell}}>\abs{\nu_{k}}.
\end{equation}
{\bfseries Step 1 - An equation for Lamé moduli:}
In view of Corollary \ref{c1}, \eqref{c3b} is true for $k=0$ and we need to show this identity for $k\geq1$.
We proceed by iteration.
Fix $p\in\{0,\ldots,N-2\}$ and assume that
\eqref{c3b} holds true for all $k=0,1,\ldots,p$ and all $\lambda\in[-R,R]^n$. Notice that this condition implies that
$$D_\eta^kC^1(\lambda,0) = D_\eta^kC^2(\lambda,0),\quad k=0,1,\ldots,p,\ \lambda\in[-R,R]^n.$$
Therefore, repeating the arguments in the proof of Theorem \ref{t4}, for all $B\in\mathcal{W}_{k,\ell}^{p+1}$, we obtain
\bel{c3c}(D_\eta^{p+1}C^1(\lambda,0)(B,\ldots,B):B)\nu = (D_\eta^{p+1}C^2(\lambda,0)(B,\ldots,B):B)\nu.\ee
From \eqref{eq-corollary-3-1} we see that, for $m=1,2$,
\begin{equation*}
\begin{aligned}
	&(D_\eta^{p+1}C^m(\lambda,0)(B,\ldots,B):B)\nu\\
	&=\operatorname{tr}(B)D_\eta^{p+1}\Lambda^m(\lambda,0)(B,\ldots,B)\nu+2D_\eta^{p+1}\mu^m(\lambda,0)(B,\ldots,B)B\nu.
\end{aligned}
\end{equation*}
Therefore, applying \eqref{c3c}, for all $B\in \mathcal{W}_{k,\ell}^{p+1}$, we obtain 
\bel{c3d}\operatorname{tr}(B)D_\eta^{p+1}\Lambda(\lambda,0)(B,\ldots,B)\nu+2D_\eta^{p+1}\mu(\lambda,0)(B,\ldots,B)B\nu=0,\ee
where \( \Lambda(\lambda,0) \coloneqq \Lambda^{1}(\lambda,0)-\Lambda^{2}(\lambda,0) \) and \( \mu(\lambda,0)\coloneqq \mu^{1}(\lambda,0)-\mu^{2}(\lambda,0) \).
We consider the \( k \)- and \( l \)-components of \eqref{c3d} as a matrix equation in the two unknowns \( D_{\eta}^{p+1}\Lambda(\lambda,0)(B,\ldots,B) \) and \( D_{\eta}^{p+1}\mu(\lambda,0)(B,\ldots,B) \).
By proving that \( (0,0) \) is the unique solution of this equation we establish the identity
\bel{c3dd}D_\eta^{p+1}\Lambda(\lambda,0)(B,\ldots,B)=D_\eta^{p+1}\mu(\lambda,0)(B,\ldots,B)=0.\ee
{\bfseries Step 2 - \( \nu_{k},\nu_{\ell} \) have different sign:}
Let \( m\in\{2,\ldots,p+1\} \) be arbitrary and choose \( B=A_{1}+\ldots+A_{m}\in \mathcal{W}_{k,\ell}^{m} \) (for the definition of \( \mathcal{W}_{k,\ell}^{m} \) see \eqref{isotropic-measurement-matrices}) with \( A_{1},\ldots,A_{m}\in\mathcal{W}_{k,\ell}^{+} \).
Let \( \theta \) be the total number of the matrices \( A_{1},\ldots,A_{m} \) which are equal to \( 2B_{kk}-B_{\ell\ell} \).
Denote by \( x_{1} = D_{\eta}^{p+1}\Lambda(\lambda,0) \) and \( x_{2} = D_{\eta}^{p+1}\mu(\lambda,0) \) and write \( x = (x_{1},x_{2}) \), then the \( k \)- and \( l \)-components of \eqref{c3d} reads
\begin{equation}\label{mat}
\begin{bmatrix}
	m\nu_{k} & 2(m+\theta)\nu_{k} + s_1 \\
	m\nu_{\ell} & -2\theta\nu_{\ell} + s_2
\end{bmatrix}
\cdot x = 0,
\end{equation}
where \( s_{1},s_{2} \) are given by
\begin{equation*}
	s_{1} = \sum_{t\neq k}\alpha_{t}\nu_{t},\quad s_{2} = \sum_{t\neq \ell}\beta_{t}\nu_{t}
\end{equation*}
for some numbers \( \alpha_{1},\ldots,\alpha_{n},\beta_{1},\ldots,\beta_{n}\in\{0,\ldots,m-\theta\} \) satisfying \( \sum_{t\neq k}\alpha_{t}\leq m-\theta \) and \( \sum_{t\neq \ell}\beta_{t}\leq m-\theta \).
Notice that in view of \eqref{c3a}, for $j=1,2$, we have 
\begin{equation}\label{sj-estimate}
\begin{aligned}
	\abs{s_{1}}&\leq (m-\theta)\abs{\nu_{\ell}} \leq (m+\theta)\abs{\nu_{\ell}}, \\
	\abs{s_{2}-\beta_{k}\nu_{k}}&\leq (m-\theta)\abs{\nu_{\ell}} \leq (m+\theta)\abs{\nu_{\ell}}.
\end{aligned}
\end{equation}
To see that the unique solution of \eqref{mat} is \( x = (0,0) \), consider the determinant
\begin{equation*}
	D\coloneqq \operatorname{det}
\begin{bmatrix}
	m\nu_{k} & 2(m+\theta)\nu_{k} + s_1 \\
	m\nu_{\ell} & -2\theta\nu_{\ell} + s_2
\end{bmatrix}.
\end{equation*}
Note that \( \nu_{k}\nu_{\ell}<0 \), so that \( \nu_{k}\nu_{\ell} = -\abs{\nu_{k}}\cdot\abs{\nu_{\ell}} \).
Using this together with \eqref{sj-estimate}, \eqref{c3a}, and \eqref{eq-k-larger} we estimate
\begin{equation*}
\begin{aligned}
	D &= m\nu_{k}(-2\theta\nu_{k} + s_{2}) - m\nu_{\ell}(2(m+\theta)\nu_{k} + s_{1}) \\
	&=2m\theta\abs{\nu_{k}}\cdot\abs{\nu_{\ell}}+m\nu_{k}s_{2} + 2m(m+\theta)\abs{\nu_{k}}\cdot\abs{\nu_{\ell}} - m\nu_{\ell}s_{1} \\
	&=2m\theta\abs{\nu_{k}}\cdot\abs{\nu_{\ell}}+m\beta_{k}\nu_{k}^{2} + m\nu_{k}(s_{2}-\beta_{k}\nu_{k}) + 2m(m+\theta)\abs{\nu_{k}}\cdot\abs{\nu_{\ell}} - m\nu_{\ell}s_{1} \\
	&\geq2m\theta\abs{\nu_{k}}\cdot\abs{\nu_{\ell}}+m\nu_{k}(s_{2}-\beta_{k}\nu_{k}) + 2m(m+\theta)\abs{\nu_{k}}\cdot\abs{\nu_{\ell}} - m\nu_{\ell}s_{1} \\
	&\geq 2m\theta\abs{\nu_{k}}\cdot\abs{\nu_{\ell}}+m(m+\theta)\abs{\nu_{k}}\cdot\abs{\nu_{\ell}} - m(m+\theta)\abs{\nu_{\ell}}^{2} \\
	&\geq 2m\theta\abs{\nu_{k}}\cdot\abs{\nu_{\ell}}+m(m+\theta)\abs{\nu_{\ell}} (\abs{\nu_{k}}- \abs{\nu_{\ell}}) >0.
\end{aligned}
\end{equation*}
Thus, \( x=(0,0) \) is  the unique solution of the system \eqref{mat} and we deduce that \eqref{c3dd} is fulfilled for \( B\in \mathcal{W}_{k,\ell}^{p+1} \) and, by polarization (see e.g. \cite[Theorem 1]{T}), we obtain \eqref{c3b} for \( k=p+1 \) and all \( \lambda\in[-R,R]^n \).
~\newline
{\bfseries Step 3 - \( \nu_{k},\nu_{\ell} \) have the same sign:}
Let \( m\in\{2,\ldots,p+1\} \) be arbitrary and choose \( B = A_{1}+\ldots+A_{m}\in \mathcal{W}_{k,\ell}^{m} \) with \( A_{1},\ldots,A_{m}\in\mathcal{W}_{k,\ell}^{-} \), and let \( \theta \) be the total number of the matrices \( A_{1},\ldots,A_{m} \) which are equal to \( 2B_{\ell\ell}-B_{kk} \).
Instead of \eqref{mat}, we now arrive at the equation
\begin{equation}\label{mat2}
\begin{bmatrix}
	m\nu_{k} & -2\theta\nu_{k} + s_2 \\
	m\nu_{\ell} & 2(m+\theta)\nu_{\ell} + s_1
\end{bmatrix}
\cdot x = 0,
\end{equation}
where \( s_{1},s_{2} \) now are given by
\begin{equation*}
	s_{1} = -\sum_{t\neq k}\alpha_{t}\nu_{t},\quad s_{2} = -\sum_{t\neq \ell}\beta_{t}\nu_{t}.
\end{equation*}
In this case we have the estimates
\begin{equation}\label{sj-estimate-2}
\begin{aligned}
	\abs{s_{1}}&\leq (m-\theta)\abs{\nu_{\ell}} \leq (m+\theta)\abs{\nu_{\ell}}, \\
	\abs{s_{2}+\beta_{k}\nu_{k}}&\leq (m-\theta)\abs{\nu_{\ell}} \leq (m+\theta)\abs{\nu_{\ell}}.
\end{aligned}
\end{equation}
Now consider the determinant
\begin{equation*}
	D\coloneqq \operatorname{det}
\begin{bmatrix}
	m\nu_{k} & -2\theta\nu_{k} + s_2 \\
	m\nu_{\ell} & 2(m+\theta)\nu_{\ell} + s_1
\end{bmatrix},
\end{equation*}
which we estimate, by using \eqref{sj-estimate-2} and \eqref{eq-l-larger}, as follows
\begin{equation*}
\begin{aligned}
	D &= m\nu_{k}[2(m+\theta)\nu_{\ell}+s_{1}]-m\nu_{\ell}(-2\eta\nu_{k}+s_{2}) \\
	&\geq 2m\theta\abs{\nu_{k}}\abs{\nu_{\ell}} + m(m+\theta)\abs{\nu_{k}}(\abs{\nu_{\ell}}-\abs{\nu_{k}})>0.
\end{aligned}
\end{equation*}
Thus, \( x=(0,0) \) is again the unique solution of the system \eqref{mat2}.
This completes the proof of the theorem.
\end{proof}

\appendix
\section{Complementary results}\label{appendix-0}

In this section we give the statement of our main results to the situation when the relation between the stress tensor and the strain tensor is given by \eqref{stress11}.
The proofs are omitted, as they are identical to the proofs in the previous sections.
Let \( \Omega \subset\mathbb{R}^{n} \) satisfy assumption \ref{assumption-domain} and let \( C(\lambda,\nu) \) be a tensor satisfying \eqref{assumption-ellipticity}.
Then, for any $\lambda\in\mathbb R^n$, the equilibrium conditions without any body forces, applied to the relation \eqref{stress11}, leads to the following  quasilinear Lam\'e system
\begin{equation}\label{eq100}
\begin{cases}
	\operatorname{div}(C(u,\nabla u):\nabla u) = 0&\text{in }\Omega, \\
	u = \lambda+f&\text{on }\partial\Omega.
\end{cases}
\end{equation}
Proposition \ref{lemma-solution-map} generalizes to this equation, which establishes the existence of a locally defined smooth solution map \( f\mapsto u_{\lambda,f} \) with which the local displacement-to-traction map
\begin{equation*}
	\mathcal{N}_{\lambda,C}(f) = [C(u_{\lambda,f},\nabla u_{\lambda,f}):\nabla u_{\lambda,f}]\nu\vert_{\partial\Omega}
\end{equation*}
can be defined.
In this setting we obtain the follows results.

\begin{theorem}\label{t6}
Let \( \Omega \) satisfy assumption \ref{assumption-domain} and, for $m=1,2$, let \( C^m \) satisfy assumptions \ref{assumption-nonlinearity}, with $k=1$,  and \ref{assumption-ellipticity}.
Let $\lambda\in\R^n$,  \( \mathcal{V} = \{x_{1},\ldots x_{n}\}\subset\partial\Omega \) be points satisfying \eqref{eq-normal-vector-condition} and \( \mathcal{W} = \{E_{ij}: i,j= 1,\ldots,n\} \).
If for some \( \delta>0 \) it holds that
\begin{equation*}
	[\mathcal{N}_{\lambda,C^{1}}(tf_{B})](x) = [\mathcal{N}_{\lambda,C^{2}}(tf_{B})](x)\quad\  B\in \mathcal{W},\ x\in\mathcal{V},\  \abs{t}<\delta,
\end{equation*}
 then $C^1(\lambda,0) = C^2(\lambda,0)$.
\end{theorem}

The uniqueness result of Theorem \ref{t6} can be extended to the following stability estimate.
\begin{theorem}\label{t8}
Suppose the assumptions of Theorem \ref{t6} holds, but now let \( \lambda \) belong to some \( n \)-cube \( \mathcal{Q} = [-R,R]^{n}\subset\mathbb{R}^{n} \) for \( R>0 \) arbitrary.
Then
\begin{equation*}
	\sup_{\lambda\in \mathcal{Q}}\norm{ C^1(\lambda,0)-C^2(\lambda,0)}_{*} \leq C\sup_{\substack{B\in\mathcal{W}\\ \lambda\in\mathcal{Q}\\ x\in\mathcal{V}}}\norm{[D\mathcal{N}_{\lambda,C^{1}}(0)f_{B}](x)-[D\mathcal{N}_{\lambda,C^{2}}(0)f_{B}](x)}_{\mathbb{R}^{n}},
\end{equation*}
where \( C \) is given by
\begin{equation*}
	C = \frac{\frac{n(n+1)}{2}\sqrt{2n(1+nM^{2})}}{1-nM^{2}}, \quad M = \max_{i=1,\ldots,n}\{\norm{\nu(x_{i})-e_{i}}_{\mathbb{R}^{n}}\}.
\end{equation*}
Above, \( \norm{\cdot}_{*} \) denotes the operator norm of fourth order tensors viewed as linear transformations on matrices.
\end{theorem}

Let us consider the subset $\mathcal{X}^{\ell} $ of $\mathbb R^{n\times n}$ given by
\begin{equation*}
	\mathcal{X}^{\ell}=\{E_1+\ldots+E_k\colon k=1,\ldots,\ell,\ E_1,\ldots,E_k\in\mathcal{X} \},
\end{equation*}
where
\begin{equation*}
	\mathcal X=\{E_{ij}:\ i,j=1,\ldots,n\}\cup \{E_{ij}+E_{k\ell}:\ i,j,k,\ell=1,\ldots,n,\ (i,j)\neq(k,\ell)\}.
\end{equation*}
For the dependency with respect to $\eta$, of the elastic tensor $C$, we obtain the following. 
\begin{theorem}\label{t7}
Let \( \Omega \) satisfy assumption \ref{assumption-domain} and let $N\geq2$.
For $m=1,2$, let \( C^m \) be tensors satisfying \ref{assumption-nonlinearity}, with $k=N$, as well as assumption \ref{assumption-ellipticity}.
Assume additionally that \( C^{m} \) take the form
\begin{equation*}
	C^m_{ijk\ell}(\lambda,\eta)=
	\begin{cases}
		\mu_{ij}^m(\lambda,\eta)D^m_{ijk\ell}(\lambda)&\text{if }(i,j)= (k,\ell),\\
		D^m_{ijk\ell}(\lambda)&\text{if }(i,j)\neq (k,\ell),
	\end{cases}
\end{equation*}
where \( (\lambda,\eta)\in \mathbb{R}^{n}\times \mathbb R^{n\times n} \) and
\begin{equation*}
\begin{aligned}
	\mu^{m} &= (\mu_{ij}^m)_{i,j}\in C^{N+1}(\mathbb R^n\times \mathbb R^{n\times n}, \mathbb R^{n\times n}), \\
	D^{m} &= (D_{ijk\ell}^m)_{ i,j,k,\ell}\in C^{N+1}(\mathbb R^n,\mathbb R^{n\times n\times n\times n}),
\end{aligned}
\end{equation*}
and \( \mu_{ij}^m(\lambda,0)=1 \) for all $i,j=1,\ldots,n$. 
Let \( \lambda\in\mathbb{R}^{n} \) and  \( \mathcal{V} = \{x_{1},\ldots x_{n}\}\subset\partial\Omega \) be points satisfying \eqref{eq-normal-vector-condition}.
If, for $R>0$ and \( \delta>0 \), it holds that
\begin{equation*}
	[\mathcal{N}_{\lambda,C^{1}}(tf_{B})](x) = [\mathcal{N}_{\lambda,C^{2}}(tf_{B})](x),\quad B\in \mathcal{X}^{N-1},\ x\in\mathcal{V},\  \abs{t}<\delta,\ \lambda\in[-R,R]^{n},
\end{equation*}
 then $D_{\eta}^kC^1(\lambda,0) = D_{\eta}^kC^2(\lambda,0),\quad k=0,1,\ldots,N-1,\ \lambda\in[-R,R]^n$.
\end{theorem}

\section{The linear equation}\label{appendix-1}
For the sake of completeness we establish in this appendix the well-posedness of the linear constant coefficient system in \( W^{2,p}(\Omega, \mathbb{R}^{n}) \).
Let \( \Omega\subset\mathbb{R}^{n} \), \( n\geq 2 \) be an open domain with $C^2$ boundary and consider the elliptic system in \( \Omega \),
\begin{equation}\label{eq-appendix-quasilinear-tensor}
\begin{cases}
	\operatorname{div}(C:\nabla u) = 0&\text{in }\Omega, \\
	u = f&\text{on }\partial\Omega,
\end{cases}
\end{equation}
where \( u\colon \Omega\to\mathbb{R}^{n} \), \( C\in \mathbb{R}^{n\times n\times n\times n} \) and \( f\colon\partial\Omega\mapsto\mathbb{R}^{n} \).
Equation \eqref{eq-appendix-quasilinear-tensor} can equivalently be written using matrix coefficients \( a_{ij}\in\mathbb{R}^{n\times n} \) as follows
\begin{equation}\label{eq-appendix-quasilinear-matrix}
\begin{cases}
	\sum_{i}\partial_{x_{i}}(\sum_{j}a_{ij}\partial_{x_{j}}u) = 0&\text{in }\Omega, \\
	u = f&\text{on }\partial\Omega.
\end{cases}
\end{equation}
The tensor \( C \) and matrices \( a_{ij} \) are related by \( C_{ijk\ell} = a_{j\ell,ik} \) where \( a_{j\ell,ik} \) is the entry at row \( i \) and column \( k \) of the matrix \( a_{j\ell} \).
Alternatively,
\begin{equation*}
\begin{aligned}
	C:B =
\begin{bmatrix}
\tilde{C}_{11}:B & \cdots &\tilde{C}_{1n} : B \\
\vdots & \ddots & \vdots \\
\tilde{C}_{n1} : B & \cdots &\tilde{C}_{nn} : B
\end{bmatrix},
\end{aligned}
\end{equation*}
where \( \tilde{C}_{ij}=\begin{bmatrix}a_{j1,i}^{T} & a_{j2,i}^{T}\ldots & a_{jn,i}^{T}\end{bmatrix} \) and \( a_{j\ell,i} \) is row \( i \) of \( a_{j\ell} \), and \( \tilde{C}_{ij}:B \) denotes the Frobenius inner product between matrices.

There are several common notions of ellipticity for systems, but most relevant for the following discussion is:
\begin{itemize}
	\item {\bfseries Strongly elliptic:} There exists a constant \( c>0 \) such that
\begin{equation}\label{eq-strongly-elliptic}
	\sum_{ij}\sum_{k\ell}C_{ijk\ell}\xi_{\ell}\eta_{k}\xi_{j}\eta_{i} = \xi^{T}\eta:C:\xi^{T}\eta \geq c\norm{\xi}_{\mathbb{R}^{n}}^{2}\norm{\eta}_{\mathbb{R}^{n}}^{2}, \quad  \xi,\eta\in\mathbb{R}^{n},\ x\in\overline{\Omega},
\end{equation}
where \( \xi^{T}\eta \) is the matrix whose entry at row \( i \) and colum \( j \) is \( \xi_{j}\eta_{i} \).
\end{itemize}
In the variable coefficient setting, strongly elliptic operators are precisely those that are coercive on \( H^{1}_{0}(\Omega,\R^n) \) \cite[Theorem 4.6]{mclean00}, and therefore enjoy a good solvability theory and in particular satisfy the Fredholm alternative \cite[Theorem 4.10]{mclean00}.
For constant coefficient equations, strong ellipticity instead lead to well-posedness in \( H^{1}_{0}(\Omega,\mathbb{R}^{n}) \) \cite[Corollary 3.46]{gm12}.
Moreover, elliptic interior regularity holds for strongly elliptic operators \cite[Proposition A.1]{jls17}.
The ellipticity condition used in \cite{jls17} is that for every \( \xi\in\mathbb{R}^{n}\setminus\{0\} \), the operator
\begin{equation*}
	\mathbb{R}^{n}\ni \eta\mapsto \sum_{i,j}a_{ij}\xi_{i}\xi_{j}\eta\in \mathbb{R}^{n}
\end{equation*}
is injective, where \( a_{ij} \) are the matrices appearing in \eqref{eq-appendix-quasilinear-matrix}.
This holds in particular under the above notion of strong ellipticity, since it follows from \eqref{eq-strongly-elliptic} and
\begin{equation*}
	\sum_{j,\ell}\inner{a_{j\ell}\xi_{\ell}\eta}{\xi_{j}\eta}_{\mathbb{R}^{n}} = \sum_{i,j,k,\ell}C_{ijk\ell}\xi_{\ell}\eta_{k}\xi_{j}\eta_{i}
\end{equation*}
that
\begin{equation*}
	\sum_{i,j}a_{ij}\xi_{i}\xi_{j}\eta = 0 \implies \|\eta\|_{\mathbb R^n} = 0.
\end{equation*}
A common alternative notion of ellipticity is obtained by expressing the positivity in \eqref{eq-strongly-elliptic} with respect to all symmetric matrices, and leads to well-posedness in \( H_{0}^{1}(\Omega,\mathbb{R}^{n}) \) by a standard Lax-Milgram argument also in the variable coefficient setting.
This is commonly used in the theory of elasticity; see for example \cite[Chapter 29]{gurtin81}.
For a short discussion on the relation between various other notions of ellipticity, see the end of Section 9.1 in \cite{rr04} or \cite[Section 3.4.1]{gm12}.
\begin{lemma}\label{lemma-linear-well-posedness}
Let \( \Omega \) satisfy assumption \ref{assumption-domain}, let \( C\in \mathbb{R}^{n\times n\times n\times n} \) be strongly elliptic, and \( 2\leq p < \infty \).
Then, for every \( g\in W^{2-1/p,p}(\partial\Omega,\mathbb{R}^{n}) \) and \( h\in L^{p}(\Omega,\mathbb{R}^{n}) \), there exists a unique solution \( u\in W^{2,p}(\Omega, \mathbb{R}^{n}) \) of the system \eqref{eq-appendix-quasilinear-tensor}.
\end{lemma}
A tensor \( \tilde{C} \) defined in terms of a tensor \( C \) by
\begin{equation*}
	\tilde{C}_{ijk\ell} = \frac{1}{2}(C_{ijk\ell} + C_{ij\ell k})
\end{equation*}
satisfies the symmetry \( \tilde{C}_{ijk\ell} = \tilde{C}_{ij\ell k} \) and it holds that
\begin{equation*}
	\operatorname{div}(C:\epsilon(u)) = \operatorname{div}(\tilde{C}:\nabla u).
\end{equation*}
The following corollary is therefore an immediate consequence of the previous lemma.
\begin{corollary}\label{corollary-linear-well-posedness-symmetric}
Let \( \Omega \) satisfy assumption \ref{assumption-domain}, let \( C\in \mathbb{R}^{n\times n\times n\times n} \) be strongly elliptic, and \( 2\leq p < \infty \).
Then, for every \( g\in W^{2-1/p,p}(\partial\Omega,\mathbb{R}^{n}) \) and \( h\in L^{p}(\Omega,\mathbb{R}^{n}) \), there exists a unique solution \( u\in W^{2,p}(\Omega, \mathbb{R}^{n}) \) of the system
\begin{equation*}
\begin{cases}
	\operatorname{div}(C:\varepsilon(u)) = 0&\text{in }\Omega, \\
	u = f&\text{on }\partial\Omega.
\end{cases}
\end{equation*}
\end{corollary}
\begin{proof}[{\normalfont\bfseries Proof of Lemma \ref{lemma-linear-well-posedness}}]
Well-posedness for weak solutions in \( W^{1,2}(\Omega, \mathbb{R}^{n}) \) is established in \cite[Corollary 3.46]{gm12}.
It only remains to establish higher regularity.
Proposition A.1 in \cite{jls17} establishes interior regularity.
We prove regularity up to the boundary using and argument from \cite[Section 2.4.3]{troianiello87}, by reduction to interior regularity by a reflection across the boundary \( \partial\Omega \).
But some of the more tedious details of various coordinate transformations are left out.

Let \( U_{h} \) be a finite open covering of \( \overline{\Omega} \) of small enough sets that each boundary section \( U_{h}\cap\partial\Omega \) can be flattened.
Let \( \psi_{h} \) be a smooth partition of unity subordinate to \( U_{h} \).
If \( u\in H^{1}_{0}(\Omega,\mathbb{R}^{n}) \) is satisfies
\begin{equation*}
	\int_{\Omega}[C:\nabla u]:\nabla \varphi\,dx = \inner{F}{\varphi},
\end{equation*}
then \( v = \psi_{h}u \) satisfies
\begin{equation}\label{eq-linear-well-posedness-1}
	\int_{U_{h}\cap\Omega}[C:\nabla v]:\nabla\varphi\,dx = \inner{F}{\psi_{h}\varphi} +\inner{G}{\varphi},
\end{equation}
where \( G \) is given by
\begin{equation*}
	\inner{G}{\varphi} = \int_{\Omega}\sum_{ijk\ell}C_{ijk\ell}\Big(u_{k}\partial_{x_{\ell}}\psi_{h}\partial_{x_{j}}\varphi_{i} - \varphi_{i}\partial_{x_{\ell}}u_{k}\partial_{x_{\ell}}\psi_{h}\Big)\,dx.
\end{equation*}
But \( u_{\ell}\in H^{1} \) and \( \psi_{h} \) is smooth, so we may integrate by parts in the first term.
Since \( \psi_{h} \) vanishes close to \( \partial U_{h}\cap \Omega \) and \( u \) vanishes on \( \partial\Omega \), this introduces no boundary terms and we obtain
\begin{equation*}
	\inner{G}{\varphi} = -\int_{\Omega}\sum_{ijk\ell}C_{ijk\ell}\varphi_{i}\Big(\partial_{x_{j}}u_{k}\partial_{x_{\ell}}\psi_{h} + u_{k}\partial_{x_{\ell}x_{j}}^{2}\psi_{h}+\partial_{x_{\ell}}u_{k}\partial_{x_{\ell}}\psi_{h}\Big)\,dx.
\end{equation*}
Hence \( G \) can be identified with a function in \( L^{2}(\Omega, \mathbb{R}^{n}) \).

Let \( \Phi\colon U_{h}\cap \Omega\mapsto V_{+}\subset\mathbb{R}^{n} \) where \( U_{h}\cap\partial\Omega \) is mapped onto a flat portion of \( \partial V_{+} \).
Note that \( \operatorname{det}(D\Phi) = 1 \); see for example the discussion in \cite[Appendix C.1]{evans10}.
Using this boundary-flattening coordinate transformation in \eqref{eq-linear-well-posedness-1} leads to
\begin{equation*}
	\int_{V_{+}} [\hat{C}(x):\nabla\hat{v}]:\nabla\hat{\varphi}\,dx = \inner{\hat{F}}{\hat{\psi}_{h}\hat{\varphi}} +\inner{\hat{G}}{\hat{\varphi}},
\end{equation*}
where
\begin{equation*}
	\hat{C}_{ijk\ell}(x) = \sum_{qp} C_{iqkp}[\partial_{p}\Phi_{\ell}(y)\partial_{q}\Phi_{j}(y)]\vert_{y=\Phi^{-1}(x)}
\end{equation*}
and the quantities \( \hat{v},\hat{\varphi},\hat{\psi}_{h} \) are defined as the composition of the corresponding functions on \( U_{h}\cap \Omega \) with \( \Phi^{-1} \).
The coefficients in \( \hat{F},\hat{G} \) are transformed similarly to \( \hat{C}(x) \) above.
Note that the transformed tensor \( \hat{C} \) is strongly elliptic whenever \( C \) is; see for example \cite[Exercise 4.2]{mclean00}; the scalar case is done in step 9 of the proof of Theorem 4 in Section 6.3.2 in \cite{evans10} and the same argument applies to the vector-valued case.
Moreover, the coordinate transformation \( \Phi \) is \( C^{1,1} \) and the coefficients \( \hat{C}_{ijk\ell} \) are therefore Lipschitz continuous.

Next, the solution \( \hat{v} \) as well as \( \hat{F} \) and the coefficients in \( \hat{G} \) are extended across the flat part of the boundary \( \partial V_{+} \) by odd reflection,
\begin{equation*}
	\tilde{v}(x',x_{n}) = -\hat{v}(x',-x_{n}).
\end{equation*}
The tensor \( \hat{C}_{ijk\ell} \) and the test functions \( \hat{\varphi} \) are extended by even reflection
\begin{equation*}
	\tilde{\varphi}(x',x_{n}) = \hat{\varphi}(x',-x_{n}).
\end{equation*}
Denote all extended quantities by \( \tilde{v},\tilde{\varphi}, \tilde{F} \) and so on.
Let \( V_{-} \) denote the domain obtained by reflecting \( V_{+} \) across the flat part of the boundary, and define \( V = \overline{V}_{+}\cup \overline{V}_{-} \) to be the entire extended domain.
Then the equation obtained in \( V \) is
\begin{equation*}
\begin{aligned}
	\int_{V}\nabla [\tilde{C}(x):\tilde{v}]:\nabla\tilde{\varphi}\,dx = &\inner{\tilde{F}}{\tilde{\varphi}}+\inner{\tilde{G}}{\tilde{\varphi}} \\
	&+ \inner{g}{\tilde{\varphi}\vert_{\{x_{n}=0\}\cap V}} + \inner{\tilde{g}}{\tilde{\varphi}\vert_{\{x_{n}=0\}\cap V}},
\end{aligned}
\end{equation*}
where \( \{x_{n}=0\}\cap V \) is the flattened part of the boundary and \( g \) is the conormal derivative coming from \( V_{+} \) and \( \tilde{g} \) is the conormal derivative coming from \( V_{-} \); see \cite[Lemma 4.3]{mclean00} for the conormal derivative.
Because of the precise locations of the minus signs in the extended functions it follows by a change of variables in the integrals that
\begin{equation*}
	\inner{g}{\tilde{\varphi}\vert_{\{x_{n}=0\}\cap V}} + \inner{\tilde{g}}{\tilde{\varphi}\vert_{\{x_{n}=0\}\cap V}} = 0.
\end{equation*}
Hence the equation reduces to
\begin{equation*}
	\int_{V}\nabla [\tilde{C}(x):\tilde{v}]:\nabla\tilde{\varphi}\,dx = \inner{\tilde{F}}{\tilde{\varphi}}+\inner{\tilde{G}}{\tilde{\varphi}}.
\end{equation*}
Just as for \( \hat{C} \), one can verify that \( \tilde{C} \) is strongly elliptic and Lipschitz continuous in the domain \( V \).
Note that \( \tilde{F},\tilde{G} \) are as regular as the original functions \( F,G \), since they are related by a \( C^{1,1} \) boundary flattening transformation and a \( C^{\infty} \) tranformation which reflects across the boundary.

We conclude by applying the interior regularity result \cite[Proposition A.1]{jls17} a finite number of times and in the end obtain \( u\in W^{2,p}. \)
To spell out some of the details, note that when \( u\in H^1 \) then \( G\in L^2 \) and therefore \( \tilde{v}\in H^2_{loc} \) by interior regularity.
Recall that \( \tilde{v} \) is defined using the partition of unity \( \psi_h \).
By repeating this for every function \( \psi_h \) we can conclude that the original solution \( u \) belongs to \( H^2(\Omega,\mathbb{R}^n) \).
Repeating the procedure, we obtain from the Sobolev embedding that \( G\in L^{s} \) for \( s=\frac{2n}{n-2} \) and conclude from the interior regularity that \( \tilde{v}\in W^{2,s_1}_{loc} \) and \( u\in W^{2,s_1}(\Omega,\mathbb{R}^n) \) where \( s_{1}=\min\{s,p\} \).
Repeating this leads to a finite sequence \( s_1<s_2<\ldots<s_m \) given by
\begin{equation*}
    s_m = \frac{2n}{n-2m}.
\end{equation*}
When \( m>\frac{n-2}{2} \) then the solution \( u\in W^{2,s_m}(\Omega,\mathbb{R}^n) \) has derivatives in \( L^{r}(\Omega,\mathbb{R}^{n}) \) for every \( r<\infty \).
Then both \( F,G \) belong to \( L^{p} \) and one final application of the regularity result yields \( \tilde{v}\in W^{2,p}_{loc} \) and \( u\in W^{2,p}(\Omega,\mathbb{R}^{n}) \).
\end{proof}
\subsection*{Acknowledgements}
D. Johansson was supported by Research Council of Finland (Center of Excellence in Inverse Modeling and Imaging and FAME Flagship, 312121 and 359208) and by the Independent Research Fund Denmark (grant 10.46540/3120-00003B). The work of Y. Kian is supported by the French National Research Agency ANR and Hong Kong RGC Joint Research Scheme for the project IdiAnoDiff (grant ANR-24-CE40-7039).
\subsection*{Conflict of interest}
The authors have no conflicts of interest to declare that are relevant to this article.
\subsection*{Ethical statement}
This study was conducted in accordance with all relevant ethical guidelines and regulations.
\subsection*{Informed Consent}
All participants provided informed consent prior to taking part in the study.
\subsection*{Data availability statement}
Data sharing is not applicable to this article as no datasets were generated or analysed during the current study.

\printbibliography
\end{document}